\documentclass{amsart}
\usepackage[dvips]{color}
\usepackage{graphicx}
\usepackage{epsfig}
\usepackage{tikz-cd}
\usepackage{enumerate}

\usepackage{latexsym}
\usepackage{amssymb}
\usepackage{amsmath}
\usepackage{amsthm}
\usepackage{amscd}

\theoremstyle{plain}
\newtheorem{thm}{Theorem}[section]
\newtheorem*{thm*}{Theorem}
\newtheorem{lem}[thm]{Lemma}
\newtheorem{defin}[thm]{Definition}
\newtheorem{prop}[thm]{Proposition}
\newtheorem*{prop*}{Proposition}
\newtheorem{cor}[thm]{Corollary}

\theoremstyle{definition}
\newtheorem{rmk}[thm]{Remark}

\newtheorem*{claim*}{Claim}

\newtheorem{ques}{Question}

\newcommand{\NN}{\mathbb{N}}
\newcommand{\ZZ}{\mathbb{Z}}

\newcommand{\RR}{\mathbb{R}}

\newcommand{\Mcg}{\mathrm{Mod}}
\newcommand{\PML}{\mathcal{PMF}}
\newcommand{\ML}{\mathcal{MF}\setminus\{0\}}
\newcommand{\MF}{\mathcal{MF}\setminus\{0\}}
\newcommand{\PMF}{\mathcal{PMF}}

\newcommand{\UE}{\mathcal{UE}}
\newcommand{\Cob}{\mathcal{COB}}

\begin{document}
\title[Path-connectivity of $\UE$]{Path-connectivity of the set of uniquely ergodic and cobounded foliations}
\author{Jon Chaika and Sebastian Hensel}
\begin{abstract}
  We show that if $S$ is a closed surface of genus $g\geq 5$ or a
  surface of genus $g\geq 2$ with at least $p\geq 1$ marked points,
  then the set of uniquely ergodic foliations and the set of cobounded
  foliations is path-connected and locally path-connected.
\end{abstract}
\maketitle

\section{Introduction}
Projective measured foliations play a prominent role in Teichm\"uller
theory, dynamics and the study of mapping class groups. In addition to
the structure of individual foliations, the set $\PMF(S)$ of all
foliations on a given finite type surface $S$ has particular
importance. $\PMF(S)$ carries a natural
(weak-$\ast$) topology and is homeomorphic to a sphere of dimension
$6g+2p-7$ if $S$ has genus $g$ and $p$ punctures. One reason for its
importance stems from the fact that 
$\PMF(S)$ can be identified with both the sphere of directions, and
the boundary of infinity of Teichm\"uller space. One can also use
$\PMF(S)$ to describe the Gromov boundary of the curve graph.

In this article we study global topological properties of two
dynamically motivated subsets of $\PML(S)$. The first is the set
$\UE(S)$ of \emph{uniquely ergodic} foliations, where a foliation $F$
is called uniquely ergodic if it admits a unique transverse measure up
to scale.  The second is the set $\Cob(S)$ of \emph{cobounded
  foliations}, where $F$ is called cobounded if a Teichm\"uller
geodesic ray with vertical foliation $F$ projects into a compact set
of the moduli space of Riemann surfaces.

These sets have been intensely studied from a dynamical point of view,
owing to their importance in Teichm\"uller theory.  As a starting
point, by a theorem of Masur \cite{Masur-Criterion}, any cobounded
foliation is uniquely ergodic, and we therefore have
\[ \Cob(S) \subset \UE(S) \subset \PMF(S) \] Both $\Cob(S)$ and
$\UE(S)$ are dense in $\PMF(S)$ (but the same is also true for their
complements). Masur and Veech
\cite{Masur-Criterion, VeechFullMeasure} show that $\UE(S)$ has full
measure in $\PMF(S)$. In contrast, the set $\Cob(S)$ has measure zero.

It is known that there are many embedded circles in $\Cob(S)$
\cite{LS-thick}. On the other hand, Masur and Smillie \cite{MasurSmillie} have
shown that the complement $\PMF(S)\setminus\UE(S)$ has Hausdorff
dimension stictly bigger than $\dim\PMF(S)-1$, and hence one
cannot expect to naively locally deform paths in order to avoid
$\PMF(S)\setminus\UE(S)$ by general position arguments.

\smallskip Our main result shows that paths are nevertheless abundant in $\Cob(S)$ and $\UE(S)$:
\begin{thm}\label{thm:main-intro}
  Let $S$ be a closed surface of genus at least $5$, or a surface of
  genus at least $2$ with at least $1$ puncture. Then the
  subsets $\UE(S), \Cob(S)$ are path-connected, and locally path-connected. Moreover for any finite set $F$ we have that $\UE(S)\setminus F$, and $ \Cob(S) \setminus F$ are path connected.
\end{thm}
In fact, the proof shows something slightly stronger: any two points
in $\UE(S)$ can be joined by a continuous path which is contained in
$\Cob(S)$ except possibly at its endpoints.

\smallskip Our result can also be used to show that through any finite
number of points in $\UE$ or $\Cob$ there is an embedded circle in $\UE$ or $\Cob$. To ensure that the circle is embedded, one has to use the proof of Theorem~\ref{thm:main-intro} rather than just the statement. We omit details, as the claim is not central to our discussion.

\subsection*{Proof Strategy and Structure of this Article} 
To build our paths, we will connect a `nice' 
(in the case of surfaces with punctures: stable foliation for a point-push Pseudo-Anosov) $p$ to an arbitrary uniquely ergodic foliation 
$\lambda$ by a sequence of paths $\gamma_0, \, \gamma_1,...$ so that 
\begin{enumerate}
\item the initial point on $\gamma_0$ is $p$.
\item The initial point on $\gamma_{j+1}$ is the terminal point of $\gamma_j$.
\item For every $\epsilon>0$ there exists $k$ so that $\gamma_j$ is contained in an $\epsilon$-neighborhood of $\lambda.$
\item $\bigcup_j\gamma_j \subset \Cob$.
\end{enumerate}
These conditions give that the concatenation of the $\gamma_j$ extends to a path from $p$ to $\lambda$ (in particular it is continuous at $\lambda$). 

We now highlight two main ingredients to accomplish this. On the one hand, we will develop in
Section~\ref{sec:pointpushs} a robust mechanism to construct paths of
cobounded foliations in the sphere of projective measured foliations
of a punctured surface. This construction was heavily inspired by the
work in \cite{LS-connectivity} (who showed that
 there is a dense path connected set of arational foliations in $\PMF$), and our main new contribution here is
to use bad approximability of points under straight line flows on tori
to certify coboundedness and to improve the paths built in
\cite{LS-connectivity} to consist of cobounded foliations. This will
be done in Section~\ref{sec:pointpushs}. 

Our second ingredient is to show how to link these paths to arbitrary uniquely ergodic foliations.
Here, we use train track splitting
sequences to define mapping class group sequences that exhibit
contracting behaviour on $\PMF$, extending the contraction of the polyhedra of measures along the splitting sequence. 
The main technical work to make this
work happens in Section~\ref{sec:train-track-north-south}, and uses
the hyperbolic geometry of curve graphs to show that these sequences
act on $\PMF$ in a contracting way.

Section~\ref{sec:islands-point-push} then combines these two parts and
shows the path-connectivity statement in Theorem~\ref{thm:main-intro}
for punctured surfaces. This is also the prerequisite for
Section~\ref{sec:closed-case}, in which the path-connectivity
statement of Theorem~\ref{thm:main-intro} is proved for closed surfaces.

Finally, in Section~\ref{sec:localpath} we show how to leverage the
constructions of paths to show local path-connectivity.

\subsection*{Further Questions}
Finally, we want to highlight a few questions for further research
suggested by Theorem~\ref{thm:main-intro} and its proof.

\begin{ques}
  Are $\UE(S)$ and $\Cob(S)$ simply connected, if the genus of $S$ is sufficiently large?
\end{ques}

\begin{ques}[{Gabai \cite{Gabai}}]
  Is the set $\mathcal{AF}(S)\supset\mathcal{UE}(S)$ of arational foliations path-connected?
\end{ques}
This question came up in Gabai's analysis of connectivity properties
of the Gromov boundary of the curve graph (which is the quotient of
$\mathcal{AF}(S)$ by the map which ``forgets'' the measure on the
foliation). Gabai proves that this boundary is path-connected, but his
methods does not apply to $\mathcal{AF}(S)$ directly.  Leininger and
Schleimer \cite{LS-connectivity} proved that the set $\mathcal{AF}(S)$
of arational foliations is connected, and contains a dense
path-connected subset, but it is not clear that these paths can be
extended to the closure. We suspect that our curve graph methods can
recover Gabai's result that ending lamination space is path connected
in the case of a surface of genus at least 5. Partly because such
genus bounds would not be optimal, we do not prove this here. However,
we want to remark that this kind of strategy is used in \cite{BCH} to
prove path connectivity and local path connectivity of the boundary of
the free factor graph.

Our methods are at the moment also unable to deal with the case of
arational foliations, mainly because the contraction properties in
Section~\ref{sec:train-track-north-south}. This is due to the fact
that in order to certify contraction we use the curve graph boundary,
which is unable to distinguish different measures supported on a
topological foliation.

\medskip Next, one could consider more restrictive subsets of $\Cob(S)$. Namely,
suppose we fix a constant $\epsilon>0$. Call a foliation $F$
$\epsilon$-cobounded if a Teichm\"uller ray with vertical foliation
$F$ eventually stays in the $\epsilon$-thick part of Teichm\"uller
space. 
\begin{ques}\label{ques:cob}
  Is the set $\Cob_\epsilon(S)$ of $\epsilon$-cobounded foliations
  path-connected for any choice of small enough $\epsilon$?
\end{ques}
Our methods do not yield this, since the paths (both in
Section~\ref{sec:islands-point-push} and~\ref{sec:closed-case}) need
to degenerate very close to simple closed curves in order to apply the
methods from Section~\ref{sec:train-track-north-south}. However, the
basic paths from Section~\ref{sec:pointpushs} can be guaranteed to
have uniform thickness.

\medskip Finally, one motivating reason for studying paths of
cobounded paths in the sphere of projective measured foliation stems
from one of the central open questions in the study of mapping class
groups and Teichm\"uller theory. Namely, Farb--Mosher
\cite{FarbMosherConvex} define \emph{convex cocompact subgroups} in
analogy to such Kleinian groups. At this time, all known examples of
such groups are virtually free, and it is not clear if any other
examples can exist. One touchstone question is therefore: is there a
convex cocompact subgroup of the mapping class group, which is
isomorphic to the fundamental group of a higher genus surface. Such a
group $G$ would give rise to a $G$--invariant circle in $\Cob(S)$.
\begin{ques}
  Are there embedded circles in $\Cob(S)$ which are invariant under groups
  that are not free?
\end{ques}
Most likely, this question requires significant new tools.  A 
weaker version of this question arises if we relax the invariance condition, e.g.
\begin{ques}
  Is there a finite subset $F\subset\Mcg(S)$, and $P \subset F^2$ so that for $x$ in Teichm\"uller space
we have that the limit in $\PML$ of
\[\{s_n...s_1x: (s_1...s_n)\in F^n \text{ and } (s_i,s_{i+1})\in P
\text{ for all }i<n\}_{n\in \mathbb{N}}\] is a circle in $\Cob(S)$.  That is, is there a ``convex
cocompact shift of finite type'' which has a circle limit set of cobounded foliations in $\PML$?  
\end{ques}
One could also ask a similar question for semigroups.
  
\section{Contractions on $\PMF$}
\label{sec:train-track-north-south}

We denote by $\PMF$ the sphere of projective measured foliations.
Recall that a foliation is called \emph{minimal}, if every regular
leaf is dense. As mentioned in the introduction, we call a foliation $F$
\emph{uniquely ergodic}, if $F$ admits a unique transverse measure
up to scale. We call a foliation $F$ \emph{cobounded} if a Teichm\"uller
ray with vertical foliation $F$ is contained in some thick part of 
Teichm\"uller space. By Masur's criterion \cite{Masur-Criterion}, cobounded foliations are uniquely ergodic,
and it is well known that uniquely ergodic foliations are minimal.

Throughout this article, we will use the notion of 
measured foliations, although most literature on train tracks uses 
measured geodesic laminations instead. We refer the reader to \cite{Levitt} for an excellent
dictionary between foliations and laminations on surfaces. 
Most of the time this will not be cause for confusion. We only want to 
emphasise that a minimal foliation in our sense corresponds to a  
minimal \emph{and filling} lamination. In particular, there are no simple 
closed curves which have intersection $0$ with a minimal foliation.

\subsection{From splitting sequences to mapping classes}
This section sets out the framework connecting mapping class group
elements and train track splitting sequences. We refer the reader to
\cite{PennerHarer} for a detailed treatment of the basic theory of
train tracks, and \cite{MasurMinsky} for some other concepts we use.

If $\tau$ is a train track and $F$ is a foliation, we write
$F\prec\tau$ if \emph{$F$ is carried by $\tau$} (compare
\cite[Section~1.6]{PennerHarer}, noting that in \cite{PennerHarer} the
notion of measured geodesic laminations is used in place of
foliations). We denote
by $P(\tau) \subset \MF$ the set of measured foliations which are
carried by $\tau$. When it does not cause confusion, we will often
identify $P(\tau)$ with the subset of the sphere $\PML$ of projective
measured foliations it defines. The set $P(\tau)$ naturally has the structure
of a closed polyhedron, whose faces correspond to the polyhedra
$P(\eta)$ of subtracks $\eta$ of $\tau$.

A train track is called \emph{recurrent}, if for every branch there is
a train path which traverses it. It is called \emph{birecurrent} if in
addition there is a multicurve hitting the train track efficiently (i.e. without 
generating bigons)
which intersects every branch (compare \cite[Section~1.3]{PennerHarer}
for details on these definitions). From now on, we will usually assume without
mention that all train tracks we use are birecurrent.  We say that a train
track is \emph{large} if every complementary component is simply
connected, and \emph{maximal}, if every complementary component is a
triangle (which implies largeness).

For maximal, birecurrent train tracks $\tau$, the interior of
$P(\tau)$ defines an open set in $\PML$
\cite[Lemma~3.1.2]{PennerHarer}. For other train tracks this need not be
the case. By the \emph{interior $\mathrm{int}\, P(\tau)$ of
  $P(\tau)$} we will always mean the subset of $P(\tau)$ formed by all those
measures which assign a positive weight to each branch. We stress
again that, in general, this is different from the topological
interior of $P(\tau)$ as a subset of $\PML$ or $\ML$.

Given a train track $\tau$, a branch $b$ is \emph{large}, if every
train path through either of its endpoints runs through $b$.
Recall that we can perform a \emph{left, right
  or central split} at a large branch to obtain a new train track
$\tau'$. Compare \cite[\S 2.1]{PennerHarer} for details on this
construction. We recall that a left or right split does not affect
the number and type of complementary components of the train track, while
a central split can join two complementary components into one.

Let $\tau$ be a fixed maximal, birecurrent train track. As noted above,
the polyhedron $P(\tau)$ defines an open set in $\PML$.  We let
$\mathcal{T}(\tau)$ be the set of all large birecurrent train tracks
which can be obtained from $\tau$ by any number of splits (left,
right, or central). The
set $\mathcal{T}(\tau)$ can be stratified in the following way. Put
$\mathcal{T}_0(\tau) = \{\tau\}$, and inductively define
$\mathcal{T}_{n+1}(\tau)$ to be the set of large train tracks obtained from
each $\sigma\in\mathcal{T}_n(\tau)$ by splitting each large branch
once (in one of the up to three possible ways). Note that a central
split need not yield a large train track, so not all three possibilities are
always allowed.

A large branch $b$ of a large birecurrent train track $\sigma$ defines a
hyperplane $H$ in $P(\sigma)$ cutting $P(\sigma)$ into subpolyhedra
$P_l, P_r$, which are exactly the polyhedra of the left and right
splits of $\sigma$. The polyhedron of the central split of $\sigma$ at $b$
is the hyperplane $H$ \cite[Proposition~2.2.2]{PennerHarer}.  Hence,
the interiors of the polyhedra $P(\sigma), \sigma \in
\mathcal{T}_n(\tau)$ define a decomposition of $P(\tau)$ into disjoint 
subpolyhedra.

Now, let $F\in\mathrm{int}\, P(\tau)$ be given, and
let
\[ \mathcal{T}(\tau, F) = \{ \sigma \in \mathcal{T}(\tau),
  F \prec \sigma\} \] be the subset of all those train tracks in
$\mathcal{T}(\tau)$ which carry $F$.  We let
$\mathcal{T}_n(\tau, F)$ be the set of all those
$\sigma \in \mathcal{T}_n(\tau)$ which carry $F$.
For the next lemma, we use the notion of \emph{diagonal extension}. If $\tau$ is a 
train track, then we say that $\eta$ is a diagonal extension of $\tau$ if $\eta$ is obtained by
adding branches inside simply connected complementary components. See \cite[Section~4.1]{MasurMinsky} for details.
\begin{lem}\label{lem:diagonal}
  The sets $\mathcal{T}_n(\tau, F)$ only contain diagonal
  extensions of the (large) train track $\eta_n\in\mathcal{T}_n(\tau, F)$ with
  the fewest complementary components.
\end{lem}
\begin{proof}
  Consider the sequence $\eta_k$ of train tracks obtained by splitting
  $\tau$ in the direction of $F$ and always choosing a central split
  when possible.  These have the property that they always carry $F$,
  and additionally, the weight defined by $F$ is positive on every
  branch of $\eta_k$ for all $k$ ($F$ \emph{fills} $\eta_k$).  Note
  that for any $\sigma\in \mathcal{T}_n(\tau)$ there exists (at least
  one) $\eta_k$ (depending on $\sigma$) so that $\eta_k$ is a subtrack
  of $\sigma$ (this follows inductively, since if a foliation is
  carried by, and fills, a subtrack $\eta$ of $\sigma$ and $\sigma$
  splits to $\sigma'$, then either $\eta$ or a split of $\eta$ is a
  subtrack of $\sigma'$). Since $F$ is minimal, and therefore there is
  no simple closed curve that doesn't intersect $F$, it can only be
  carried by large train tracks. Therefore, $\sigma$ is a diagonal extension
  of $\eta_k$. The lemma now follows, since any set
  $\mathcal{T}_n(\tau, F)$ contains at most one $\eta_k$, and the
  number of complementary components in a split decreases only during
  central splits.
\end{proof}
We put
\[ U_n(\tau, F) = \bigcup_{\sigma \in \mathcal{T}_n(\tau,
    F)} \mathrm{int}\, P(\sigma) \]
\begin{lem}\label{lem:splitting-neighbourhoods}
  $U_n(\tau, F)$ is an open neighborhood of $F$ in $\PMF$ for every $n$.
\end{lem}
\begin{proof}
  We prove the lemma by induction. For $n=0$ this is simply openness
  of $P(\tau)$ (\cite[Lemma~3.1.2]{PennerHarer}).
  Suppose now that $U_n(\tau, F)$ is an open neighborhood of
  $F$. From the description of the effect of splits on polyhedra given
  above we conclude that $U_{n+1}(\tau, F)$ is obtained from $U_n(\tau, F)$
  by cutting at hyperplanes (corresponding to central splits) and retaining
  those polyhedra which contain $F$. If none of these hyperplanes contain $F$,
  it is clear that $U_{n+1}(\tau, F)$ is still an open neighbourhood of $F$.
  However, suppose that one of them does contain $F$. This corresponds to the
  situation in which a track $\eta \in \mathcal{T}_n$ has a large branch so that
  all three of the left,right and central splits of $\eta$ along that branch carry $F$ -- and therefore all three of these splits will contribute to $U_{n+1}(\tau, F)$, guaranteeing that the latter is still an open neighbourhood of $F$.
\end{proof}

A \emph{splitting sequence $\tau_i$ in the direction of $F$} is
a sequence $\tau_i$ of train tracks with $\tau_0 = \tau$ and so that
each $\tau_i$ carries $F$, and $\tau_{i+1}$ is obtained from $\tau_i$
by splitting exactly one large of $\tau_i$ branch once. A \emph{full splitting sequence} instead
requires splitting each large branch of $\tau_i$ once when passing from $\tau_i$ to $\tau_{i+1}$.
Hence, if $\tau_i$ is a full splitting sequence in the direction of $F$ starting in $\tau$, then
$\tau_i \in \mathcal{T}_i(\tau)$ for all $i$.

If $F$ is a foliation in the minimal stratum (i.e. each singularity is
$3$--pronged, and there are no saddle connections\footnote{This means that the corresponding lamination has only triangles as its complementary components.}), then each split in
a splitting sequence $\tau_i$ is a left or a right split, and
furthermore the type is uniquely determined by $F$.  If
$F$ has saddle connections or $k$--prong singularities for $k > 3$, then it is possible that for some $n$,
$F$ is carried by the left, right and central split of
$\tau_n$. This is furthermore the last time $F$ is carried in
the interior of a maximal train track $\tau_n$ along the splitting sequence.

Splitting sequences in the direction of minimal foliations have good
contracting properties. In the following theorem, and below, we denote
by $\Delta(F)$ the (closed) simplex of projective measured foliations
which are topologically equivalent to $F$.
\begin{thm}[{compare e.g. \cite[Theorem~5.1.1]{Mosher}}]\label{thm:mosher}
  Suppose that $F$ is a minimal foliation 
  and that $\tau_i$ is any splitting sequence in the direction of $F$. Then
  \[ \bigcap_{i=1}^{\infty} P(\tau_i) = \Delta(F). \]
\end{thm}
As an immediate corollary, we have
\begin{cor}\label{cor:u-nest}
  Let $F \in\mathrm{int}\, P(\tau)$ be minimal. Then
  \[ \bigcap_{n\geq 0} U_n(\tau,F) = \Delta(F). \]
\end{cor}
We now describe how to connect splitting sequences to sequences in the
mapping class group. The first step is the following lemma.
\begin{lem}\label{lem:track-to-mcg}
  There is a finite number of sets
  \begin{equation}\label{eq:basic patches}
    \mathcal{T}^{(1)}, \ldots, \mathcal{T}^{(M)} 
  \end{equation}
  so that for each maximal train track $\tau$, each minimal $F$, 
  and each $n$ there is a number $k_{\tau, F, n}$ and a mapping
  class $f_{\tau, F, n}$ with
  \[ \mathcal{T}_n(\tau, F) =
    f_{\tau,F,n}\left(\mathcal{T}^{(k_{\tau, F, n})}\right). \] The number $k_{\tau, F, n}$ is unique. The mapping class $f_{\tau, F, n}$ is unique up to a
  finite indeterminacy.
\end{lem}
We call the set of $\mathcal{T}^{(i)}$ \emph{standard neighborhood models}
and we call the number $k_{\tau, F, n}$ the \emph{type} of $\mathcal{T}_n(\tau, F)$.
\begin{proof}[Proof of Lemma~\ref{lem:track-to-mcg}]
  By Lemma~\ref{lem:diagonal}, 
   $\mathcal{T}_n(\tau, F)$
 consists of train tracks which are diagonal
  extensions of some large train track $\eta_n$. Since the mapping class
  group $\Mcg(S_g)$ acts on the set of (isotopy classes of) train
  tracks on $S_g$ with finitely many orbits, there are finitely many
  choices for such a train track $\eta_n$ up to the mapping class
  group action. Since the number of complementary components of $\eta_n$
  can be bounded from the Euler characteristic of $S$ alone, there are a finite number of diagonal extensions of $\eta_n$. This implies that the 
  mapping class group also
  acts on the sets $\mathcal{T}_n(\tau, F)$ (over all
  $\tau, F, n$) with finitely many orbits. We can therefore
  choose the sets $\mathcal{T}^{(i)}$ to be orbit representatives of
  this action. This shows both the desired existence of
  $k_{\tau, F, n}$ and $f_{\tau, F, n}$, as well as the
  uniqueness of $k_{\tau, F, n}$. The (coarse) uniqueness of the
  $f_{\tau, F, n}$ follows since the set of mapping classes
  which fix a given train track is finite (compare
  e.g. \cite[Lemma~4.2]{Ursula-TrainTrackGraph}), and so the element
  $f_{n,F}$ is also determined up to a finite choice.
\end{proof}
Let $(\tau_i)$ be a full splitting sequence starting in a maximal
train track $\tau$ towards some minimal foliation $F$. Then
each $\tau_i \in \mathcal{T}(\tau, F)$, and in fact $\tau_i \in
\mathcal{T}_i(\tau, F)$.
We then get an \emph{associated $\Mcg$-sequence} $(f_i,k_i)$ by applying
Lemma~\ref{lem:track-to-mcg} to  $\mathcal{T}_i(\tau, F)$ for each $i$.
In particular, we then have
\[ \mathcal{T}_n(\tau, F) = f_n(\mathcal{T}^{(k_i)}). \] As
before, the numbers $k_i$ are uniquely determined by the splitting
sequence, and the mapping classes $f_n$ are determined up to a finite
choice. We call the number $k_n$ \emph{the type of the index
$n$}.

Let $\mathcal{U}^{(k)}$ be the neighborhoods associated to
our standard models $\mathcal{T}^{(k)}$, i.e.
\begin{equation}\label{eq:def cal U}
 \mathcal{U}^{(k)} = \bigcup_{\sigma\in\mathcal{T}^{(k)}} \mathrm{int }P(\sigma). 
\end{equation}
We call the $\mathcal{U}^{(k)}$ the \emph{standard
  neighbourhoods}\footnote{Note that the model neighbourhoods
  $\mathcal{U}^{(k)}$ need not be contained in the polyhedra
  $P(\tau_i)$ along the splitting sequence.}.  By the defining
property of the associated sequence $(f_n, k_n)$ we can then relate
the standard neighbourhoods to the neighbourhoods of $F$ given by the
splitting sequence in the following way:
\begin{equation}
  \label{eq:relation-u-up-down}
  U_n(\tau, F) = f_n\left(\mathcal{U}^{(k_n)}\right).
\end{equation}
The next lemma collects two crucial properties of the associated sequence.
\begin{lem}\label{lem:split-means-finite}
  There is a finite set $M \subset \Mcg(S_g)$ with $M=M^{-1}$ and so
  that the following holds.  Suppose that $f_n, f_{n+1}$ are two
  consecutive terms of an associated $\Mcg$-sequence.
  Then we have
  \begin{equation}
    \label{eq:finite-increments}
    f_n^{-1}f_{n+1} \in M.
  \end{equation}
  Furthermore, 
  \begin{equation}
    \label{eq:crossnesting}
    f_n^{-1}f_{n+1}\left(\mathcal{U}^{(k_{n+1})}\right) \subset \mathcal{U}^{(k_n)}.
  \end{equation}
  \end{lem}
\begin{proof}
  Let $T$ be the (finite) set of all those train tracks which can be
  obtained from one of the train tracks in $\cup_i\mathcal{T}^{(i)}$ by full
  splits, and let $M_0$ be the set of all those mapping classes which map train 
  tracks $\sigma \in T$ to train tracks in any $\cup_j\mathcal{T}^{(j)}$. Note
  that since $T$ is finite, $M_0$ is finite by
  Lemma~\ref{lem:track-to-mcg}. We put $M = M_0 \cup M_0^{-1}$.

  To see that it has property~(\ref{eq:finite-increments}), observe that if $f_n,
  f_{n+1}$ are consecutive terms of an associated $\Mcg$-sequence, there are train tracks $\eta_n
  \in \mathcal{T}^{(i_n)}, \eta_{n+1} \in \mathcal{T}^{(i_{n+1})}$, so that
  $f_{n+1}\eta_{n+1}$ is a full split of $f_n\eta_n$. This implies that
  $f_n^{-1}f_{n+1}\eta_{n+1}$ is a full split of $\eta_n$.

  In other words, $f_n^{-1}f_{n+1}$ maps a train track in $\mathcal{T}^{(i_{n+1})}$ to
  one in $T$, and is therefore an element of $M$ by definition.

  \smallskip Equation~(\ref{eq:crossnesting}) follows immediately from the following:
  \[ f_{n+1}\left(\mathcal{U}^{(k_{n+1})}\right) = U_{n+1}(\tau, F) \subset U_n(\tau, F) = f_n\left(\mathcal{U}^{(k_n)}\right).\]
\end{proof}
The \emph{type-$k$ subsequence} is the maximal subsequence
$f^{(k)}_s = f_{r_s}$ so that $k_{r_s} = k$. We say that type $k$ is
\emph{essential for the splitting sequence $(\tau_i)$}, if the
subsequence $f^{(k)}_s$ is an infinite sequence. At least one type is
essential, but we suspect that
the type of the initial train track need not repeat infinitely often. 

By Lemma~\ref{lem:split-means-finite} we have
that $f_{i+1}f_i^{-1} \in M$ for all $i$; but we warn the reader that
the elements $f^{(k)}_{s+1}\left(f^{(k)}_{s}\right)^{-1}$ are not constrained to
a finite set in the mapping class group.

\subsection{Minimal Foliations and the Curve Graph}
\label{subsec:curve-graph-machine}
In this section we will prove that large terms in the associated
$\Mcg$-sequence for a uniquely ergodic foliation send certain subsets
of $\PML$ into small neighborhoods of the foliation, and will use this
to prove contracting properties for associated $\Mcg$-sequences.
Intuitively, we will show that all curves (and non-minimal foliations)
are attracted to the foliations $F$ guiding the splitting sequence,
and we will show that the speed of attraction can be controlled for
certain geometrically constrained sets of curves.

\smallskip We begin by rephrasing the contraction exhibited by train
track polyhedra under splitting sequences (Theorem~\ref{thm:mosher})
in terms of associated $\Mcg$-sequences.
\begin{cor}\label{cor:south-dynamics}
  Let $\tau_i$ be a splitting sequence towards a minimal foliation $F$
  and let $(f_i, k_i)$ be an associated $\Mcg$-sequence.
  For any essential type $k$ we have that
  \[ \bigcap_s f^{(k)}_s(\mathcal{U}^{(k)}) = \Delta(F). \] 
\end{cor}
\begin{proof}
  By Corollary~\ref{cor:u-nest}, we have that
  \[ \bigcap_{n\geq 0}U_n(\tau, F) = \Delta(F), \]
  and therefore, by definition of essential type,
  \[ \bigcap_{n\geq 0, k_n=k}U_n(\tau, F) = \Delta(F). \]
  Now, using Equation~(\ref{eq:relation-u-up-down}) we see that
  $U_n(\tau, F) = f_n(\mathcal{U}^{(k)})$ if the index $n$ is of
  type $k$, and therefore
  \[ \bigcap_{n\geq 0, k_n=k}U_n(\tau, F) = \bigcap_s f^{(k)}_s(\mathcal{U}^{(k)}),\]
  which shows the corollary.
\end{proof}
In other words, the mapping classes $f^{(k)}_s$ eventually 
contract $\mathcal{U}^{(k)}$ to a small neighborhood of $\Delta(F)$. 
The rest of this section is concerned with studying the contraction
properties of the mapping classes $f_i$ outside the open sets
$\mathcal{U}^{(k)}$. 

To this end, we use the geometry of the curve graph. Recall that the
\emph{curve graph} $\mathcal{C}(S)$ of a surface is the graph whose
vertex set is the set of isotopy classes of essential simple closed
curves on $S$, with edges between classes that admit representatives
with intersection $0$. We denote by $d_{\mathcal{C}(S)}$ be the resulting metric on $\mathcal{C}(S)$. The core feature of the geometry of the curve
graph we need is the following.
\begin{thm}[{Masur-Minsky \cite{MasurMinsky}}]\label{thm:cg-hyperbolicity}
  If $S$ is a non-exceptional surface (i.e. $\mathcal{C}(S)$ is connected), then the
  curve graph is hyperbolic in the sense of Gromov.
\end{thm}
We will need two methods to produce quasigeodesics in the curve
graph. The first one is the method employed to show hyperbolicity in
\cite{MasurMinsky}.
\begin{thm}\label{thm:tm-shadows}
  Let $S$ be a surface of finite type. Then there are numbers
  $K,K'$, depending on $S$ with the following property: suppose
  that $\rho:\mathbb{R}\to\mathcal{T}(S)$ is a Teichm\"uller geodesic,
  and suppose that for each $t\in \mathbb{R}$ the curve $\alpha_t$ has
  smallest possible extremal length\footnote{See e.g. \cite{Ahlfors} for a definition. The precise definition of this does not matter too much to understand the theorem; it would remain true also for e.g. the shortest hyperbolic geodesic on $\rho(t)$.} on $\rho(t)$. Then the assigment
  \[ t \to \alpha_t \] is an unparametrised $K$--quasigeodesic in the
  curve graph. In particular, for any $t<s$, the set
  $\{\alpha_r, t\leq r\leq s\}$ has Hausdorff distance at most $K'$ from
  a curve graph geodesic joining $\alpha_t$ to $\alpha_s$.
\end{thm}
\begin{proof}
   Theorem~2.3 of \cite{MasurMinsky} states
   that a coarsely transitive path family with the contraction property
   in a geodesic metric space consists of uniform unparametrised
   quasigeodesics (for the definitions, compare Section~2.4 of
   \cite{MasurMinsky}).  Theorem~2.6 of \cite{MasurMinsky} then shows
   that the family of paths in the curve graph obtained by taking
   shortest extremal length curves has the contraction property (that
   these paths are coarsely transitive is easy to see).
\end{proof}
The second, related construction of quasigeodesics uses train tracks.
It is proven in \cite[Theorem~1.3]{MM3}, see also
\cite[Corollary~2.6]{Ursula-boundary}:
\begin{prop}\label{prop:split-shadows}
  Let $S$ be a surface of finite type. Then there are numbers $K,K'$,
  depending on $S$ with the following property: suppose that
  $(\tau_i)_i$ is a splitting sequence and suppose that for each
  $i\in \mathbb{N}$ the curve $\alpha_t$ is a vertex cycle\footnote{See e.g. \cite{PennerHarer} or \cite[Section~4.1]{MasurMinsky} for a definition of vertex cycle. Again, the precise definition of this does not matter too much; the theorem would remain true for e.g. the shortest train path on $\tau_i$.} on
  $\tau_i$. Then the assigment
  \[ i \to \alpha_i \] is an unparametrised $K$--quasigeodesic in the
  curve graph. In particular, for any $t<s$, the set
  $\{\alpha_r, t\leq r\leq s\}$ has Hausdorff distance at most $K'$ from
  a curve graph geodesic joining $\alpha_t$ to $\alpha_s$.
\end{prop}
For a Gromov hyperbolic space, one can define a boundary at
infinity, see e.g. \cite[III.H.3]{BH} for details.
If $\alpha_0$ is some basepoint, recall the \emph{Gromov product}
\[ (x\cdot y)_{\alpha_0} = \frac{1}{2}(d(\alpha_0, x) + d(\alpha_0, y)
- d(x,y)).\] A sequence $(x_i)_i$ of points in $X$ \emph{converges at
  infinity} if $(x_i\cdot x_j)_{\alpha_0} \to \infty$ as $i,j\to
\infty$. The Gromov boundary is then defined as a set of equivalence
classes of sequences converging to infinity, where two sequences
$(x_i),(y_j)$ are equivalent if $(x_i\cdot y_j)_{\alpha_0} \to \infty$
as $i,j\to \infty$; see \cite[III.H.3.12]{BH} for details.

Also note that the Gromov product extends from the space
to the boundary at infinity \cite[III.H.3.15]{BH}. The Gromov product has the
property that
\begin{equation}
  \label{eq:gromov-product-triangle}
   | (x\cdot y)_{\alpha_0} - (x'\cdot y)_{\alpha_0} | \leq d(x,x')
\end{equation}
for any point $y$ and (finite) points $x,x'$.

\medskip In the case of the curve graph, the Gromov boundary can be
identified explicitly with a different space.  We define the set
$\mathcal{EL}(S)$ to be the set of minimal foliations with the
measure-forgetting topology. That is, we consider the subset
$\mathcal{M} \subset \PML$ of all minimal foliations, and let
$\mathcal{EL}(S)$ be the quotient topological space $\mathcal{M}/\thicksim$ under the
equivalence relation which lets $F\thicksim F'$ if $F, F'$ are
topologically equivalent.

\begin{thm}[{\cite[Theorems~1.2,~1.3~and~1.4]{Klarreich}}]\label{thm:klarreich}
  \begin{enumerate}[i)]
	\item The Gromov boundary of $\mathcal{C}(S)$ is homeomorphic to the space
	$\mathcal{EL}(S)$.
	\item A sequence $\alpha_i$ of curves (interpreted as points in the curve
	graph) converges to the point at infinity defined by a minimal
	foliation $F$ if and only if every accumulation point of $\{\alpha_i, i\in\NN\}$
	in $\PML$ is contained in $\Delta(F)$.
	\item 
	Suppose that $\rho$ is a  Teichm\"uller geodesic ray whose vertical foliation is a minimal
	foliation $F$, and that for every $t$, the curve $\alpha_t$ is a curve of smallest extremal length on
	$\rho(t)$. Then the curves $\alpha_t$ (interpreted as points in the curve graph)
	converge to $F$ (interpreted as a point in the Gromov boundary)
  \end{enumerate}
\end{thm}

As a consequence of Theorem~\ref{thm:klarreich} we have the following characterization of neighborhoods in $\PML$ using the curve graph.
\begin{lem}\label{lem:gromov-product-criterion}
  Suppose that $F$ is a minimal foliation, and $U$ is an open
  neighborhood of $\Delta(F)$ in $\PML$. Let $\gamma$ be an arbitrary simple closed curve.
  Then there is a number $K$ with the following property: suppose that $\beta$ is a simple
  closed curve so that (as a point in the curve graph) we have
  \[ (F\cdot\beta)_\gamma > K. \]
  Then $\beta$ (seen as a projective measured foliation) is contained in $U$.
\end{lem}
\begin{proof}
  Suppose that the claim were false. Then we would find a
  sequence $(\beta_i)$ with $(\beta_i \cdot F)_{\gamma} > i$
  but $\beta_i \notin U$. By the Gromov product condition,
  $(\beta_i)$ would then be a sequence converging at infinity to the
  boundary point $F$. So, by Theorem~\ref{thm:klarreich}~ii), the
  sequence $\beta_i$ converges in the measure forgetting topology to
  $F$. Since $U$ is an open neighborhood of
  $\Delta(F)$ this is impossible as $\beta_i \notin U$.
\end{proof}
We also need the following partial converse.
\begin{lem}\label{lem:gromov-product-criterion-2}
  There is a number $k_0$, depending only on the topological type of the surface, 
  with the following property.
  Suppose that $F$ is a minimal foliation, $\alpha$ is a simple closed curve,
  and 
  \[ (F\cdot\alpha)_\gamma > B \]
  If $\mu_i$ is a sequence of minimal foliations converging to $\alpha$ in $\PML$, then
  \[ (F\cdot\mu_i)_\gamma > B-k_0 \]
  for all large $i$.
\end{lem}
\begin{proof}
  Denote by $\Phi:\mathcal{T}(S)\to\mathcal{C}(S)$ the map which
  assigns to a marked hyperbolic surface in Teichm\"uller space a
  curve of smallest extremal length\footnote{This curve may not be
    well-defined, but any two choices have uniformly few intersections
    due to the collar lemma. Hence, any two choices have uniformly
    small distance in the curve graph.}. Pick a basepoint $X_0$ in
  Teichm\"uller space for which $\gamma$ is a curve of smallest
  extremal length, and consider the Teichm\"uller geodesic rays
  $\rho_i$ starting from $X_0$ in the direction of $\mu_i$. Since the
  $\mu_i$ converge to $\alpha$ in $\PML$, the rays $\rho_i$ converge
  uniformly on compact subsets to the  Teichm\"uller
  geodesic ray $\rho_\infty$ starting in $X_0$ with vertical foliation
  $\alpha$.
  
  Theorem~\ref{thm:tm-shadows} implies that there is a
  constant $K$ (depending only on the topological type of the surface) so that the images $\Phi\circ\rho_i$ can be reparametrised
  to be $K$--quasigeodesics $q_i$ beginning in $\gamma$.
  By Theorem~\ref{thm:klarreich}~iii), the quasigeodesic $q_i$ connects
  $\gamma$ to the point $\mu_i$ in the Gromov boundary of the curve
  graph. 
 
  There is a constant $T_0$ so that $\Phi\circ\rho_\infty(t)$ is coarsely equal to $\alpha$ for all $t\geq T_0$. As the $\rho_i$ converge to $\rho_\infty$ uniformly on compact sets in Teichm\"uller
  space, one concludes that $\Phi\circ\rho_i(T_0)$ is also coarsely equal to $\alpha$ for all large $i$.
  Hence, the $q_i$ pass uniformly close by $\alpha$ for
  all large $i$.  This implies that there is a constant $k_0$, depending on $K$ and the hyperbolicity constant of the curve
  graph (and hence only the topological type of the surface), so that $(F\cdot\mu_i)_\gamma > (F\cdot\alpha)_\gamma-k_0$,
  which implies the lemma.
\end{proof}

The next lemma and corollary are well known and standard and included for completeness.
\begin{lem}\label{lem:quasiconvexity}
  Let $F$ be a minimal foliation, and $K$ a number. Then suppose that
  $x,y\in \mathcal{C}(S)$ 
   with
  \[ (F\cdot x)_\gamma, (F\cdot y)_\gamma \geq K. \]
  Let $z$ be a point on a geodesic between $x,y$. Then
  \[ (F\cdot z)_\gamma \geq K - 4\delta, \]
  where $\delta$ is the hyperbolicity constant of the curve graph.
\end{lem}
\begin{proof}
  First we observe that if $x,y,z$ are three points in $\mathcal{C}(S)$ and $z$ lies on
  a geodesic between $x$ and $y$, we have
  \begin{eqnarray*}
    2(x\cdot z)_\gamma = d(\gamma, x) + d(\gamma, z) - d(x,z)
    &\geq& d(\gamma, x) + d(\gamma, y) - d(y, z) - d(x, z)\\ 
    &=& d(\gamma, x) + d(\gamma, y) - d(x, y) = 2(x\cdot y)_\gamma.
  \end{eqnarray*}
  By $\delta$--hyperbolicity, we have that for all triples $a,b,c$ of points in
  $\mathcal{C}(S) \cup \partial_\infty\mathcal{C}(S)$
  \[ (a\cdot c)_\gamma \geq \min\{ (a\cdot b)_\gamma, (b \cdot
  c)_\gamma \} - 2\delta, \] compare
  e.g. \cite[III.H.3.17.(4)]{BH}. First, apply this to $x, F, y$ to
  conclude that
  \[ (x\cdot y)_\gamma \geq K - 2\delta. \]
  Now, apply this same estimate again, to conclude
  \begin{eqnarray*}
    (F\cdot z)_\gamma &\geq& \min\{ (F\cdot x)_\gamma, (x \cdot z)_\gamma \} - 2\delta \\
    &\geq& \min\{ (F\cdot x)_\gamma, (x \cdot y)_\gamma \} - 2\delta \\
    &\geq& \min\{ K, K-2\delta \} - 2\delta \\
    &\geq& K-4\delta
  \end{eqnarray*}
  which is what we wanted to prove.
\end{proof}
\begin{cor}\label{cor:quasiconvexity}
  Let $K,D>0$ be numbers, $F$ be a minimal foliation.  Suppose
  that $\tilde{x},\tilde{y}$ are any two 
  points in the curve complex or its boundary,  satisfying
  \[ (F\cdot \tilde{x})_\gamma, (F\cdot \tilde{y})_\gamma \geq K. \]
  Suppose that $z \in \mathcal{C}(S)$ lies 
  on a (possibly infinite)
  $D$--quasi-geodesic $q$ with endpoints $\tilde{x}$ and $\tilde{y}$.
  Then
  \[ (F\cdot z)_\gamma \geq K - X, \] where $X$ is a number
  depending only on the hyperbolicity constant of the curve graph and
  the quasi-geodesic constant $D$.
\end{cor}
\begin{proof}
  Choose points $x_i=q(r_i), y_i=q(s_i)$ in the curve complex on the
  quasi-geodesic $q$ which converge to $\tilde{x},\tilde{y}$
  respectively. If an endpoint of $q$ is finite, we assume that the
  corresponding sequence is eventually constant. Recall, e.g. from
  \cite[III.H.3.17.(5)]{BH}, that  
  \[ \lim\inf (F\cdot x_i)_\gamma \geq (F\cdot \tilde{x})_\gamma - 2\delta \]
  and 
  \[ \lim\inf (F\cdot y_i)_\gamma \geq (F\cdot \tilde{y})_\gamma - 2\delta. \]
  By our assumption, we then conclude that 
  \[ \min\{(F\cdot x_i)_\gamma, (F\cdot y_i)_\gamma\} \geq
  K-2\delta-1, \]
  for large $i$. We furthermore assume that $i$ is large
  enough so that $z$ is contained in the subsegment $q_i$ of $q$ with
  endpoints $x_i, y_i$. By $\delta$--hyperbolicity, there is a number
  $B$ depending on $D$ (and $\delta$), so that the Hausdorff distance between
  $q_i$ and the geodesic connecting $x_i$ to $y_i$ is at most $B$. Let
  $z'$ be a point on that geodesic of distance at most $B$ to $z$.
  By Lemma~\ref{lem:quasiconvexity}, we then have
  \[ (F\cdot z')_\gamma \geq K - 6\delta - 1, \]
  and thus
  \[ (F\cdot z)_\gamma \geq K - 6\delta - B - 1. \]
  Hence $X = 6\delta + B + 1$ satisfies the requirement.
\end{proof}

\begin{lem}\label{lem:gromov-increase}
  Let $F$ be a minimal foliation, $\tau$ a train track
  and $(\tau_i)$ a splitting sequence in the direction of $F$
  and let $(f_i, k_i)$ be an associated $\Mcg$-sequence.
  Suppose that $(\gamma_i)$ is a sequence of simple closed curves so that
  $\gamma_i$ is contained in $f_i^{(k_i)}(\mathcal{U}^{(k_i)})$
  for every $i$. Then, for any base point
  $\alpha_0$, we have
  \[ (\gamma_i \cdot F)_{\alpha_0} \to \infty. \]
\end{lem}
\begin{proof}
  By Corollary~\ref{cor:u-nest} and the assumption, any accumulation point of the curves
  $\gamma_i$ (interpreted as projective measured foliation) is contained in
  $\Delta(F) \subset \PML$. By Theorem~\ref{thm:klarreich}~ii), the $\gamma_i$
  therefore converge (interpreted as points in the curve graph) to
  $F$ in the Gromov boundary. By definition, this implies that
  the Gromov product condition claimed in the corollary.
\end{proof}

We can use this to show the following contraction behavior for finite-diameter subsets in the curve graph.
\begin{prop}\label{prop:curves-get-contracted}
  Let $F$ be a minimal foliation, $\tau$ a train track
  and $(\tau_i)$ a splitting sequence in the direction of $F$
  and let $(f_i, k_i)$ be an associated $\Mcg$-sequence.

  \smallskip Consider any neighborhood $\mathcal{V}$ of $\Delta(F)$ in
  $\PML$, and let a simple closed curve $\beta_0$ and a number $d>0$ be
  given.

  Then there is a number $N=N(\tau, F, \mathcal{V}, \beta_0, d) > 0$ so that the following holds: If
  $\beta$ is any simple closed curve with
  $d_{\mathcal{C}(S)}(\beta_0, \beta) \leq d$, then
  \[ f_n(\beta) \in \mathcal{V} \quad\quad\forall n > N. \]
\end{prop}
\begin{proof}
  As a first reduction, note that by Corollary~\ref{cor:u-nest} we may
  assume that $\mathcal{V}$ is of the form
  $f_s^{(k_s)}(\mathcal{U}^{(k(s))})$ for a large enough $s$.
  Fix, for concreteness, a vertex cycle $\alpha_0$ of $\tau$ as a basepoint in the curve graph (recall that there are finitely many such choices).  
  
  \smallskip Apply Lemma~\ref{lem:gromov-product-criterion} in order
  to obtain a number $D>0$ with the property that if $\gamma$ is any
  curve so that the Gromov product satisfies
  \[ (\gamma \cdot F)_{\alpha_0} > D, \] then
  $\gamma \in \mathcal{V}$ as an element of $\PML$. 

  \smallskip Now, for each $k$ choose a curve $\delta_k$ contained in
  $\mathcal{U}^{(k)}$ and put $\gamma_n = f_n(\delta_{k(n)})$.

  Observe that
  \[ d_{\mathcal{C}(S)}(f_n(\beta_0), \gamma_n) \leq \max_k d_{\mathcal{C}(S)}(\beta_0, \delta_k) = C_0 \]
  and, if $d_{\mathcal{C}(S)}(\beta, \beta_0) \leq d$ we therefore have
  \[ d_{\mathcal{C}(S)}(f_n(\beta), \gamma_n) \leq C_0 + d. \]

  Thus, using Equation~(\ref{eq:gromov-product-triangle}), we see
  \[ (f_n(\beta) \cdot F)_{\alpha_0} \geq (\gamma_n \cdot
  F)_{\alpha_0} - d_{\mathcal{C}(S)}(f_n(\beta), \gamma_n) \geq
  (\gamma_n \cdot F)_{\alpha_0} - (C_0+d). \] Applying
  Lemma~\ref{lem:gromov-increase} to the curves $\gamma_n$ we see that
  there is a number $N$ so that
  \[ (\gamma_n \cdot F)_{\alpha_0} > D + C_0 + d\quad\quad\forall n > N. \]
  Together with the previous inequality this implies that
  \[ (f_n(\beta) \cdot F)_{\alpha_0} > D\quad\quad\forall n > N, \]
  which finishes the proof. 
\end{proof}
The next lemma, which requires a definition, will allows us to obtain
that large terms in the $\Mcg$-sequence to a uniquely ergodic
foliation contract certain infinite diameter subsets of the curve
graph (thought of as foliations) to a small neghborhood of the
uniquely ergodic foliation.
\begin{defin}
  Let $D$ be a number, and $\psi$ a pseudo-Anosov map. A
  \emph{($D$--)quasi-axis} is a bi-infinite $D$--quasi-geodesic $q:\mathbb{R}\to\mathcal{C}(S)$ so that its image $\psi^jq$ has (Hausdorff)
  distance at most $D$ from the image of $q$ for any power $j\in \ZZ$.
\end{defin}
\begin{lem}\label{lem:existence-quasiaxes}
  There are constants $D,B>0$, just depending on the surface, so that every 
  pseudo-Anosov map $\psi$ of $S$ has a $D$--quasi-axis. Furthermore, 
  any two such quasi-axes have Hausdorff distance at most $B$.
\end{lem}
\begin{proof}
  Let $\rho:\RR\to \mathcal{T}(S)$ be the Teichm\"uller geodesic
  invariant under $\psi$, i.e. there is some $T>0$ so that for all $t$
  we have $\psi\rho(t) = \rho(t+T)$. For each $t\in[0,T)$, choose a
  curve $\alpha_t$ of smallest extremal length on $\rho(t)$. For
  $t\in[iT,(i+1)T)$ put $\alpha_t = \psi^i(\alpha_{t-Ti})$. Then for all
  $t$, the curve $\alpha_t$ has smallest extremal length on
  $\rho(t)$. By Theorem~\ref{thm:tm-shadows}, the assignment
  $t\to \alpha_t$ is an (unparametrised) quasigeodesic with
  quasigeodesic constant just depending on the topological type of the
  surface. By construction, $t\to \alpha_t$ is invariant under the
  action of $\psi$. This shows that quasi-axes exist.

  The uniqueness statement follows since any quasiaxis for $\psi$
  converges in the Gromov boundary of the curve graph to the stable
  and unstable foliation of $\psi$ by Theorem~\ref{thm:klarreich}~iii)
  and two $D$-quasigeodesics with the same endpoints in a Gromov
  hyperbolic space have bounded Hausdorff distance.
\end{proof}
In the future, we will choose a $D$ for which
Lemma~\ref{lem:existence-quasiaxes} holds once and for all, and simply
refer to quasi-axes of pseudo-Anosov maps.

Also recall the definition of a \emph{Dehn twist} $T_\alpha$ about a
simple closed curve $\alpha$ (compare
e.g. \cite[Section~3.1]{Primer}). If $\alpha$ is a multicurve, together with
a choice of left/right for each component, then we denote by $T_\alpha$
the product of the left/right Dehn twists about the curves in $\alpha$.
\begin{prop}\label{prop:apply-pseudo}
  Let $F$ be a minimal foliation, $\tau$ a train track
  and $(\tau_i)$ a splitting sequence in the direction of $F$
  and let $(f_i, k_i)$ be an associated $\Mcg$-sequence.
  
  Consider any neighborhood $\mathcal{V}$ of $\Delta(F)$ in $\PML$.
  Let $\psi$ be a pseudo-Anosov, and let $\alpha$ be a multicurve which is within
  distance $d$ of its quasi-axis in the curve graph. Let $r>0$ be
  any number. Suppose $\beta_0$ is a curve.

  Then there is a number
  $N=N(\tau, F, \mathcal{V}, \psi, \alpha, d, \beta_0)>0$ 
  following property.  Suppose that $n > N$ is given. Then there is a
  number $t_0$ (which depends on $n$), so that for all $t > t_0$ the
  conjugate $\hat{\psi} = (T_\alpha)^t\circ \psi\circ (T_\alpha)^{-t}$
  satisfies the following:
  
  If $\beta$ is any simple closed curve with
  $d_{\mathcal{C}(S)}(\beta_0, \beta) \leq d$, then
  \[ f_n(\hat{\psi}^j\beta) \in \mathcal{V}, \quad\quad\forall j \in \ZZ \]
\end{prop}
\begin{proof} 
  We follow a similar strategy as in the previous proposition. Choose
  $\alpha_0$ a vertex cycle of $\tau$. Apply
  Lemma~\ref{lem:gromov-product-criterion} to find a number $U$ so
  that if
  \[ (\gamma \cdot F)_{\alpha_0} > U, \] then $\gamma \in
  \mathcal{V}$ as an element of $\PML$. 

  Introduce the notation
  \[ \hat{\psi}_t = (T_\alpha)^t\circ \psi\circ (T_\alpha)^{-t}. \]
  We therefore need to show, that there is a number $N$ so that for all
  $n > N$ there is a $t_0$ so that 
  \[ ( f_n(\hat{\psi}_t^j\beta) \cdot F )_{\alpha_0} > U, \]
  for any curve $\beta$ with $d_{\mathcal{C}(S)}(\beta_0, \beta) \leq d$, and
  any $t>t_0$, any $j \in \ZZ$.

  \medskip The first stage of the proof consists of a (lengthy)
  reduction of this statement to a similar statement
  (Equation~\eqref{eq:apply-pseudo-key} below) about quasi-axes of the
  $\hat{\psi}_t$. To begin showing this reduction, note that
  \[ ( f_n(\hat{\psi}_t^j\beta) \cdot F )_{\alpha_0} \geq 
  ( f_n(\hat{\psi}_t^j\beta_0) \cdot F )_{\alpha_0} - d(\beta, \beta_0)\geq 
  ( f_n(\hat{\psi}_t^j\beta_0) \cdot F )_{\alpha_0} - d \]
  and therefore it suffices to show
  \[ ( f_n(\hat{\psi}_t^j\beta_0) \cdot F )_{\alpha_0} > U+d. \]
  Arguing as above, we have that
  \[ ( f_n(\hat{\psi}_t^j\beta_0) \cdot F )_{\alpha_0} >  ( f_n(\hat{\psi}_t^j\alpha) \cdot F )_{\alpha_0} - d(\alpha, \beta_0). \]
  Hence, it suffices to show that 
  \[ ( f_n(\hat{\psi}_t^j\alpha) \cdot F )_{\alpha_0} > U+d+d(\alpha, \beta_0) =: U_1, \]
  for any $t>t_0$, any $j \in \ZZ$.

  Now, let $\rho$ be a ($D$--)quasi-axis for $\psi$.  Since the
  mapping class group acts as isometries on the curve graph, we have
  that $f_nT_\alpha^t\rho$ is a ($D$--)quasi-axis for
  $f_n\hat{\psi}_tf_n^{-1}$. Furthermore,
  \[ d(f_n\alpha, f_nT_\alpha^t\rho) = d(\alpha, T_\alpha^t\rho) =
  d(\alpha, \rho) = A, \] for all $t$, since $T_\alpha$ acts as an
  isometry fixing $\alpha$. Hence, $f_n\alpha$ is (for all choices of
  $n$ and $t$) within $A$ of the $D$--quasi-axis $f_nT_\alpha^t\rho$
  of $f_n\hat{\psi}_tf_n^{-1}$. Let $\eta$ be a point on $f_nT_\alpha^t\rho$
  with $d(f_n\alpha, \eta) \leq A$. The $D$--quasi-axis property then
  implies that for any $j$ we have that
  \[ d( (f_n\hat{\psi}_tf_n^{-1})^j\eta, f_nT_\alpha^t\rho ) \leq D \]
  and therefore
  \[ d( (f_n\hat{\psi}_tf_n^{-1})^jf_n\alpha, f_nT_\alpha^t\rho ) \leq A+D \]
  As such, we have that
  \[ d( f_n(\hat{\psi}_t^j\alpha), f_nT_\alpha^t\rho) = 
  d( f_n(\hat{\psi}_t^jf_n^{-1} f_n\alpha), f_nT_\alpha^t\rho) = 
  d( (f_n(\hat{\psi}_tf_n^{-1})^j(f_n\alpha), f_nT_\alpha^t\rho) 
  \leq A+D. \]
  Therefore, to prove the proposition, it suffices to show that
  there is a number $N$, so that for all $n>N$ there is a number $t_0$,
  so that for all $t>t_0$: 
  \begin{equation}\label{eq:apply-pseudo-key}
  \forall x \in f_nT_\alpha^t\rho:  (x \cdot F )_{\alpha_0} > U_1+A+D =: U_2.
  \end{equation}
  
  \medskip Now, use Lemma~\ref{lem:gromov-increase} as
  in the previous proof, to find a number $N$ so that 
  \begin{equation}
    (f_n(\alpha) \cdot F)_{\alpha_0} > 2U_2 + X + k_0
    \quad\quad\forall n > N,\label{eq:apply-pseudo-key-2}
  \end{equation}
  where $X$ is the number from
  Corollary~\ref{cor:quasiconvexity} and $k_0$ is the number from
  Lemma~\ref{lem:gromov-product-criterion-2}, applied to the
  quasi-geodesic constant $D$.  At this point, fix a number $n > N$.

  Observe that if $\mu_+, \mu_-$ are the stable and unstable
  foliations of $\psi$, then $T_{\alpha}^t\mu_+, T_{\alpha}^t\mu_-$ are the stable and
  unstable foliations of $\hat{\psi}_t$. Note that as $t\to \infty$,
  both of these foliations converge to $\alpha$ in $\PML$.  Consider
  $f_n\hat{\psi}_t(f_n)^{-1}$, and observe that its stable and
  unstable foliations therefore converge to $f_n(\alpha)$ in $\PML$
  as the number $t$ increases. By
  Lemma~\ref{lem:gromov-product-criterion-2}, this implies that we can
  choose $t_0$ large enough, so that for any $t > t_0$ we have
  \[ (f_n(T_{\alpha}^t\mu_+) \cdot F)_{\alpha_0} > U_2 + X \]
  \[ (f_n(T_{\alpha}^t\mu_-) \cdot F)_{\alpha_0} > U_2 + X \]
  Let now $z$ be any point on a $D$--quasi-geodesic with endpoints
  $f_n(T^t\mu_+), f_n(T^t\mu_-)$. Then Corollary~\ref{cor:quasiconvexity} implies
  that 
  \[ (z \cdot F)_{\alpha_0} > U_2 \] Since the quasi-axis
  $f_nT_\alpha^t\rho$ is such a $D$--quasi-geodesic, the proposition follows.
\end{proof}

In the proof of local path-connectivity, we require  uniform control
over the constants $N$ appearing in the previous two results (Propositions~\ref{prop:curves-get-contracted} and~\ref{prop:apply-pseudo})
 Before
stating the corresponding lemma, suppose that $(\tau_i)_i$ is a full splitting
sequence in the direction of some minimal foliation $F$.

Then consider, in Proposition~\ref{prop:curves-get-contracted} or~\ref{prop:apply-pseudo}, a neighbourhood $\mathcal{V} = U_k(\tau,
F)$, and observe that it is also a neighbourhood of
$\Delta(E)$ for all minimal $E \in U_i(\tau, F), i\geq k$.
Additionally, $E$ determines a full splitting sequence starting in $\tau$,
whose first $i$ terms are identical with the one defined by $F$.

Hence, it makes sense to apply
Proposition~\ref{prop:curves-get-contracted} or~\ref{prop:apply-pseudo} for this neighbourhood $\mathcal{V}$, and
$E$ in place of $F$ with its full splitting sequence starting
in $\tau$. The following lemma shows a boundedness of the resulting numbers
$N$ that these propositions produce. 
\begin{lem}\label{lem:following-sequences}
  Suppose that $(\tau_i)_i$ is a full splitting sequence in the
  direction of some minimal foliation $F$ with $\tau_1 = \tau$. Put $\mathcal{V} =
  U_k(\tau, F)$ for some $k$.

  Suppose we are given either
  \begin{enumerate}
  \item A curve $\beta_0$ and a number $d>0$, or
  \item A pseudo-Ansosov $\psi$, a curve $\alpha$, a number $r>0$ and
    a curve $\beta_0$.
  \end{enumerate}
  Then there are numbers $M,N > 0$ with the property that the number
  \begin{enumerate}
  \item $N(\tau, E, \mathcal{V}, \beta_0, d)$ from
    Proposition~\ref{prop:curves-get-contracted}, or
  \item $N(\tau, E, \mathcal{V}, \psi, \alpha, d, r, \beta_0)$ from
    Proposition~\ref{prop:apply-pseudo}  
  \end{enumerate}
  can be chosen to be smaller than $N$ for all minimal $E \in U_M(\tau, F)$.
\end{lem}
\begin{proof}
  We will describe the case of Proposition~\ref{prop:apply-pseudo} in detail,
  the corresponding argument for
  Proposition~\ref{prop:curves-get-contracted} is similar and simpler.

  Recall from the proof of Proposition~\ref{prop:apply-pseudo} that
  what one needs to show is the estimate in
  (\ref{eq:apply-pseudo-key}). This in turn is implied by
  (\ref{eq:apply-pseudo-key-2}), which is purely a statement about
  Gromov product growth of images of $\alpha$ under the associated
  mapping class group sequence $f_n$ of the given splitting sequence. Hence, to finish the proof, we will 
  argue that the number $N$ in
  (\ref{eq:apply-pseudo-key-2}) can be uniformly bounded for the associated mapping class 
  sequences $f_n$ defined by splitting sequences  arising from any minimal $E\in U_M(\tau,F)$ independent of $E$ itself.

  If $E \in U_M(\tau, F)$, then by definition the first $M$
  terms of the associated $\Mcg$-sequence for $E$ and $F$
  agree.  Hence, to show this lemma, we have to show that the existence of a number $N$
  making (\ref{eq:apply-pseudo-key-2}) true can already be guaranteed
  by knowing a large initial segment of the associated $\Mcg$-sequence. The remainder of
  this proof is concerned with showing that.

  \smallskip Similar to the proof of
  Proposition~\ref{prop:curves-get-contracted}, choose for each $k$ a
  curve $\alpha_k$ which is carried by each $\sigma \in
  \mathcal{T}^{(k)}$ as a vertex cycle.
  By Proposition~\ref{prop:split-shadows}, the path $n \mapsto f_n\alpha_{k(n)}$ is then 
  uniformly Hausdorff close to a uniform quasi-geodesic in the curve graph which
  converges to $F$.
  
  In particular, this implies that for any $K_0$ there is an $N$ with the property
  that
  \[ (f_n\alpha_{k(n)}\cdot F)_\gamma > K_0 \quad\mbox{for all
    }n>N. \] If now $F' \in U_N(\tau, F)$ and $(f'_i)$ is an associated $\Mcg$-sequence for $F'$,
  then we may assume $f_i' = f_i$ for all $i \leq N$ by definition. Thus,
  for some uniform constant $c$ (depending on the quasi-geodesic constant $k_1$ Proposition~\ref{prop:split-shadows} and the hyperbolicity constant of the curve graph)
  we have that
  \[ (f'_n\alpha_{k'(n)}\cdot F)_\gamma > K_0-c \quad\mbox{for all
    }n>N. \]
  Since the distance between the curve $\alpha$ and the (finitely many) $\alpha_k$
  is bounded,
  there is a further constant $d$ so that
  \[ (f'_n\alpha \cdot F)_\gamma > K_0-c-d \quad\mbox{for all
  }n>N. \] Choosing $K_0-c-d > 2U_2+X+k_0$ then yields that the corresponding
  $N$ works in (\ref{eq:apply-pseudo-key-2}) for the sequences of all
  $F'\in U_N(\tau, F)$, proving the lemma.
\end{proof}

\section{Paths by pushing points}
\label{sec:pointpushs}
In this section we will construct many
special paths of cobounded foliations for punctured surfaces, which
will serve as building blocks for all subsequent constructions. The
paths we will eventually use to connect uniquely ergodic foliations
will be concatenations of paths of this form, except possibly at a
countable set of points which will be stable foliations of
pseudo-Anosovs (or the endpoints).

The construction described in this section is crucially inspired by the work of Leininger and Schleimer in \cite{LS-connectivity}, where they build paths of minimal foliations. Our main contribution is that we modify their construction to produce paths of uniquely ergodic 
(and in fact cobounded) foliations, and obtain some extra control over how these paths
follow a ``combinatorial skeleton" given by a finite set of curves.

\subsection{Preliminaries on Covers, and on Adding Points}
\label{sec:cover-prelims}
Our notation follows \cite{LS-connectivity} and we refer the reader 
to that article for a very good and
readable source for background information on the methods used here.

A \emph{smooth surface} will denote a smooth, connected, compact, oriented
$2$--manifold without boundary. All maps between smooth surfaces will
be assumed to be smooth unless specified. By a slight abuse of
notation, a \emph{(holomorphic) Abelian differential on $S$} is a
smooth $1$--form $\omega$ which is holomorphic with respect to some
complex structure on $S$ (compatible with orientation and smooth
structure).  We denote by $d_\omega$ the (singular) flat metric on the
surface defined by integrating $\omega$.

We let $\widetilde{\Omega}(S)$ be the set of all such Abelian
differentials. Note that $\widetilde{\Omega}(S)$ is a path-connected
set (in fact, a vector bundle over a contractible base; compare
\cite[Section~2.6]{LS-connectivity}). 

The quotient
\[ \Omega(S) = \widetilde{\Omega}(S) / \mathrm{Diff}_0(S) \]
is the Hodge bundle of Abelian differentials over Teichm\"uller space of $S$. We need
a variant for surfaces with marked points (which is, crucially, the point of this whole discussion).
Namely, if $\mathbf{z} \subset S$ is a finite, ordered set of distinct points, we let
$\mathrm{Diff}_0(S,\mathbf{z})$ denote the group of diffeomorphisms of $S$, fixing each 
point in $\mathbf{z}$, which are homotopic to the identity through such maps. We let
\[ \Omega(S, \mathbf{z}) = \widetilde{\Omega}(S) / \mathrm{Diff}_0(S,
\mathbf{z}) \] As in \cite{LS-connectivity}, the central idea is that
any Abelian differential $\omega\in\widetilde{\Omega}(S)$ defines projections 
$\hat{\omega}\in\Omega(S, \mathbf{z})$ and $\bar{\omega}\in\Omega(S)$ (in the notation of
\cite{LS-connectivity}).

There is an action of $\mathrm{SL}_2(\mathbb{R})$ on $\widetilde{\Omega}(S)$
defined in the usual way (e.g. by postcomposing canonical flat charts) which
descends to the usual $\mathrm{SL}_2(\mathbb{R})$--action on $\Omega(S)$. We
denote by $g_t$ the action of diagonal matrices, i.e. Teichm\"uller geodesic 
flow.

\subsection{Torus Covers and Badly Approximable Points}
In this section, we begin to construct Abelian differentials with desirable
horizontal foliations.

To begin, we say that a Abelian differential $\omega\in
\widetilde{\Omega}(S)$ is \emph{(eventually) $\epsilon$--thick} if
there exists $N$ so that for all $t>N$ we have that every essential
simple closed curve on $S$ has length $\geq \epsilon$ with respect to
the singular flat metric $g_t\omega$. We say that $\omega$ is
\emph{strongly (eventually) $\epsilon$--thick with respect to
  $\mathbf{z}$} if the same is true for any arc with endpoints in
$\mathbf{z}$. 
Note that (strong) eventual thickness is invariant under the $\mathrm{Diff}_0(S,\mathbf{z})$--action, and therefore the notion also makes sense 
for differentials in $\Omega(S,\mathbf{z})$. 

The purpose of this section is to give a robust
criterion that we will use to construct many paths of thick Abelian
differentials. 

We make the following (slightly idiosyncratic) definitions, which will be one of the core mechanisms in our construction.
\begin{defin}
  \begin{enumerate}[i)]
  \item Let $(X,d)$ be a metric space and $T:X \to X$ be a dynamical
    system. We say a pair of points $(x,y) \in X$ is \emph{B-badly
      approximable} if there exists $N$ so that $k\cdot d(T^kx,y)\geq
    B$ for all $k\geq N$ and moreover $T^kx \neq y$ for all $k \neq
    0$. We may also say that the point $y$ \emph{$B$-badly
      approximates $x$}.
\item We say a rotation $R_\alpha$ of the circle is
    \emph{$B$-badly approximable} if the pair $(x,x)$ is $B$-badly
    approximable for some (equivalently every) $x \in \mathbb{R}$ for the dynamical system
    \[ R_\alpha : \mathbb{R}/\mathbb{Z} \to \mathbb{R}/\mathbb{Z}, \quad\quad z \mapsto z + \alpha \; \mathrm{mod }\,\ZZ\]
   Note that the set of $\alpha$ that are $B$-badly badly approximable for some $B>0$ agrees with the usual definition of the set of badly approximable $\alpha$. 
  \item Similarly, if
    $F^t:X\to X$ is a measurable flow of a metric space, we say a pair
    of points $(x,y)$ is \emph{B-badly approximable} if there exists
    $N$ so that $t\cdot d(F^tx,y)\geq B$ for all $t\geq N$ and moreover, $F^tx\neq y$ for all $t\neq 0$. We say a straight line flow
    on a torus is \emph{B-badly approximable} if the pair $(x,x)$ is
    $B$-badly approximable for some $x$.
  \end{enumerate}
\end{defin}
The following lemma shows why we are interested in badly approximable points.
\begin{lem}\label{lem:saddle-growth}
  If $q$ and $q'$ are distinct $B$-badly approximable points on a torus then any
  trajectory $\gamma$ from $q$ to $q'$ has $|g_t\gamma|\geq \sqrt{B}$ for all
  large enough $t$.
\end{lem}
\begin{proof}Let $t_0$ satisfy $d(F^Lq,q')>\frac BL$ for all $L\geq
  t_0$.  Because $q$ and $q'$ are not in the same orbit, by the
  definition of $B$-badly approximable, $\underset{t \to \infty}{\lim}\,
  |g_t\gamma|=\infty$ for every $\gamma$ a trajectory from $q$ to $q'$. Thus, 
  we may restrict our attention to the cofinite set of such $\gamma$
  with vertical component at least $t_0$.  Let $\gamma$ be  such a
  geodesic from $q$ to $q'$. Because the torus is flat,
  if the vertical component of $\gamma$ is $a$ and the horizontal
  component is $b$ we have that $d(F^aq,q')=b$. Since we assume that
  $(q,q')$ are $B$-badly approximable, $b$ is at least $\frac{B}a$ if
  $a$ is large enough. Since the product of the horizontal and
  vertical components of curves are preserved by $g_t$, we have
  $|g_t\gamma|$ is at least $\sqrt{2ab}\geq\sqrt{B}$ for all $t$. (We
  are also using the elementary fact that the shortest vector in the
  positive cone in $\mathbb{R}^2$ with fixed product of horizontal and
  vertical components has angle $\frac \pi 4$.)
\end{proof}

\begin{defin}
  Let $S$ be a closed surface of genus $g \geq 2$. An Abelian
  differential $\omega \in \widetilde{\Omega}(S)$ is called 
  \emph{$(\epsilon,B)$--torus good with respect to marked points $q_1, \ldots, q_k$}
  if there is a regular branched cover, branched over one point,
  \[ p:S \to T \]
  of $S$ to a torus $T$ and an Abelian differential $\omega_T$ on $T$ so that
  \begin{enumerate}
  \item $\omega_T$ is eventually $\epsilon$--thick.
  \item $\omega$ is the pullback of $\omega_T$.
  \item The images $p(q_i)$ of $q_i$ in $T$ are pairwise $B$--badly approximable with
    respect to the flat structure defined by $\omega_T$.
  \end{enumerate}
  The \emph{associated data} to the $(\epsilon,B)$--torus good $\omega$ comprise
  the cover $p$ and the base differential $\omega_T$.
\end{defin}
The notion of being torus good is invariant under the action of
$\mathrm{Diff}_0(S,\{q_1,\ldots, q_k\})$ by pulling back differentials, and therefore is also
defined for differentials in $\Omega(S,\{q_1,\ldots, q_k\})$.

The following connects the above definition to Teichm\"uller dynamics.
\begin{prop}\label{prop:good-pairs-long-segments}
  For any $(\epsilon,B)$ and $S$ there is a number $\delta>0$ with the
  following property. If $\omega$ is $(\epsilon,B)$--torus good with
  respect to marked points $q_1, \ldots, q_k$, then $\omega$ is
  eventually strongly $\delta$--thick with respect to $\mathbf{z}=(q_1,\ldots, q_k)$.
  
  In particular, the horizontal foliation of $\omega$ is cobounded as a foliation on 
  $(S,\mathbf{z})$.
\end{prop}
\begin{proof}
To prove that $\omega$ is eventually strongly $\delta$-thick with respect to $\mathbf{z}$, by definition we have to show that there is a $t_0$ so that: 
\begin{itemize}
\item if $\gamma$ is a simple closed curve on $\omega$ then $|g_t\gamma|\geq \delta$ for all $t>t_0$ and 
\item if $\gamma$ is a trajectory from $q_i$ to $q_j$ with $j \neq i$ we have that $|g_t\gamma|>\delta$ for all $t>t_0$.
\end{itemize}
The first condition follows for any $\delta\leq \epsilon$ because we
are assuming that $\omega_T$ is eventually $\epsilon$-thick and any simple closed
curve on $\omega$ projects to a closed curve of the same length
on $\omega_T$ because we are branched over a single point. Similarly
we have that $\pi(\gamma)$ is a trajectory from $\pi(q_i)$ to
$\pi(q_j)$ and $\pi(g_t\gamma)=g_t\pi(\gamma)$ and so any such
trajectory has length at least $B$ by the Lemma~\ref{lem:saddle-growth}.  
This implies the two conditions above, and therefore eventual strong $\delta$--thickness of $\omega$. 

To see the second claim, note that as $t\to \infty$, the differentials
$g_t\omega$ all lie in a compact set of the moduli space of flat
surfaces by the first part.  This in turn implies that Teichm\"uller
flow in the direction of the horizontal foliation of $\omega$ also
only defines Riemann surfaces which lie in a compact set of the moduli
space of $S-\mathbf{z}$. This shows the proposition. 
\end{proof}

Next, we will show that these torus good differentials are in fact
dense in the set of all differentials.  The proof of this uses Schmidt games,
a technique from Diophantine approximation, which we briefly define
and discuss in the next section.
\subsection{Schmidt game digression}
 Suppose we are given a set $E\subset \mathbb{R}^n$.
Suppose two players Bob and Alice take turns choosing a sequence of closed Euclidean balls  $$B_0\supset A_1\supset B_1\supset A_2\supset B_2\ldots $$ (Bob choosing the $B_i$ and Alice the $A_i$) whose diameters satisfy, for fixed $0<\alpha,\beta<1$,  and all $i>0$
\begin{equation}\label{def:Schmidt}
   |B_i|=\beta |A_i| \quad\text{and}\quad |A_{i+1}|=\alpha |B_i|.  
\end{equation}
The only requirement on $B_0$ is that it has positive diameter.
Following Schmidt \cite{schmidt} we make the following definition.
\begin{defin}
We say $E$ is an $(\alpha,\beta)$-{\em winning set} if Alice has a strategy so that no matter what Bob does, $\cap_{i=1}^\infty B_i\in E$. It is $\alpha$-{\em winning} if it is $(\alpha,\beta)$-winning for all $0<\beta<1$.  $E$ is a {\em winning set for Schmidt game} if it is $\alpha$-winning for some $\alpha>0$.
\end{defin}
 

Because Bob's first move is unconstrained we have:
\begin{lem}\label{lem:dense} If $S$ is an $(\alpha,\beta)$ winning set for any $\alpha,\beta$ then $S$ is dense in $X$.
\end{lem}
\begin{lem}(\cite[Theorem 2]{schmidt}])\label{lem:intersecting}If $S_1,...,S_k$ are $\alpha$-winning sets then $\cap_{i=1}^k S_i$ is $\alpha$-winning.
\end{lem}
In fact the previous lemma is true for countable intersections as well, but we do not need this stronger statement.
\begin{thm}{\cite[first line of Section 2]{Tseng}} Let $\xi \in [0,1)$, $R$ denote rotation by $\xi$ and $x \in [0,1)$. The set of $y$ so that $x,y$ is $(\frac{\alpha\beta}{ 4})^3$-badly approximable (for $R$) is a $(\alpha,\beta)$ winning set.
\end{thm}
From the previous two results we obtain:
\begin{cor} Given any rotation $R$ and a finite number of points $p_1,...,p_r$ in $[0,1)$ we have that the set of $q$ so that $p_i,q$ are $B$-badly approximable for some $B>0$ are $\alpha$-winning for some $\alpha>0$. 
\end{cor}
By iterating the previous result and Lemma \ref{lem:dense} we get:
\begin{cor}\label{cor:schmidt} Given any rotation on $[0,1)$, the set of $p_1,...,p_k$ that are pairwise $B$-badly approximable for some $B>0$ is dense in $[0,1)^k$. 
\end{cor}

\begin{rmk}
Note that the above discussion can be modified to treat fixed $B$, as Lemma \ref{lem:intersecting} can be modified to show the finite intersection of $(\alpha,\beta)$-winning is $(\alpha,\tilde{\beta})$ for some $\tilde{\beta}$ (which depends on $\alpha$, $\beta$ the number of intersections). While this may be useful for some applications (such as answering Question \ref{ques:cob}) it plays no role in this work and so we do not discuss it here.
  
\end{rmk}
\subsection{Density of torus good differentials}
\begin{prop}\label{prop:dense-ba} 
  Let $S$ be a closed surface of genus $g \geq 2$ and
  $q_1, \ldots, q_k$ be a set of marked points. Suppose that 
  we fix a regular branched cover, branched over one point,
  \[ p:S \to T \]
  of $S$ to a torus $T$ and an Abelian differential $\omega_T$ on $T$
  so that $\omega$ is the pullback of $\omega_T$. 
  
  Then for every $\omega_T$, every neighborhood $U$ of $\omega_T$ in
  $\widetilde{\Omega}(T)$, and every $\delta>0$ there exists $\omega'_T \in U$ and points
  $q_i'\in S$ with $d_\omega(q_i,q_i')<\delta$, for all $i$, so that the pullback of
  $\omega_T'$ is $(\epsilon,B)$--torus good with respect to marked
  points $(q'_i)$. In this, $\epsilon$ can be chosen independent of $\omega$
  and $B$ can be chosen to just depend on $k$.
\end{prop}
\begin{proof}
  We will work throughout with the canonical flat charts defined by $\omega$,
  $\omega_T$ realizing $p$ as a holomorphic map. We will then show that we can
  move the $q_i$ by a small amount (in these charts!) and modify 
  $\omega_T$ by a small rotation to obtain an $(\epsilon,B)$--torus good differential. 
  This is enough to show the proposition. 

  Given a straight line flow on a flat torus, there are many
  (geodesic) transversals so that the first return map of the flow to
  the transversal is a rotation. Moreover there exists $C$ so that
  every aperiodic straight line flow on a flat torus of area $1$ has infinitely many
  transversals $\gamma$, so that the first return map to $\gamma$ is a
  rotation and the return times to $\gamma$ are between $\frac
  1{C|\gamma|}$ and $\frac{C}{|\gamma|}$.  To see this, note that
  there is a compact set $K$ in the moduli space of flat tori, so that
  if the orbit $g_t\omega_T$ of a torus $\omega_T$ under Teichm\"uller
  flow does not diverge to infinity (without recurring), 
  then there exist arbitrarily large $t$ so that $g_t\omega_T \in K$.
  In the case of an aperiodic straight line flow the 
  first case does not happen. In the second case, a side of the fundamental domain 
  of the torus  $g_t\omega_T \in K$ will work as a transversal.

\noindent
\textbf{Sublemma:} Let $p,q$ be points on a torus $T$ and $F^t$ a
minimal straight line flow on $T$.  Suppose that $\gamma$ is a
transversal for $F^t$, and let $T_0$ be the minimal first return time
of $F^t$ to the transversal $\gamma$. Assume further that the first
return of $F^t$ defines a rotation $R_\xi$ on $\gamma$.  Suppose that
$s_1, s_2>0$ are minimal so that $F^{s_1}p, F^{s_2}q \in \gamma$, and
that the points $F^{s_1}p, F^{s_2}q$ are $B$-badly approximable for
$R_\alpha$. Then $p$ and $q$ are $B'$-badly approximable for $F^t$ for
any $B' < B\cdot T_0$.
\begin{proof}[Proof of Sublemma]We prove the statement by contradiction. Assume that there exists $\epsilon>0$
  and arbitrarily large $L$ so that
$$d(F^Lp,q)<\frac{BT_0-\epsilon}L.$$
Assume that the straight line flow is vertical. We may assume that
$F^Lp$ is on the same horizontal as $q$. Let $T_1$ be the maximal
return time of the flow to $\gamma$. Then there is some $0\leq
\ell\leq 3T_1$, so that $F^\ell q \in \gamma \setminus \partial
\gamma$ (since at most two returns can hit a boundary point of $\gamma$).
Furthermore, after fixing $\ell$, we have that for all large enough $L$:
\[ d(F^Lp,q)<d(F^\ell q,\partial \gamma) \] Let $h \subset \gamma$ be
the shortest horizontal segment connecting $F^Lp$ to $q$. Then
$F^\ell(h) \subset \gamma$, and it is a horizontal segment of length $d(F^Lp, q)$ joining
$F^{\ell+L}p$ to $F^\ell q$. Since $F^{\ell+L}p$ and $F^{s_1}p$ are in the same $R_\alpha$--orbit,
there is a power $k$ so that $F^{\ell+L}p = R_\alpha^kF^{s_1}p$. Since $s_1$ is the first time that 
the flow line through $p$ hits $\gamma$, we know that $k\leq 3+ \frac L{T_0}$. In other words, for
this $k$ we have:
\[ d(R_\alpha^kF^{s_1}p,F^\ell q)=d(F^Lp,q) \]
Since rotations are isometries, and $F^\ell q$ is in the $R_\alpha$ orbit of
$F^{s_2}q$, there exists some $j\geq 0$ so that \[d(R_\alpha^{k-j}F^{s_1}p,F^{s_2}
q)=d(F^Lp,q).\]
If $L$ is large enough (depending on $\epsilon$), we then have a contradiction of our claim that $F^{s_1}p$, $F^{s_2}q$ are $B$-badly approximable.
\end{proof}

Next, observe that there exists $B>0$ so that the rotation $R_\xi$
for any $\xi$ whose continued fraction expansion terminates in all
1's is $B$-badly approximable. Note that such $\xi$ are dense in the reals.

Now, suppose we are given the torus $\omega_T$. Pick a transversal
$\gamma$ so that the first return map for the vertical straight line
flow on $\omega_T$ defines a rotation on $\gamma$, and furthermore the return time
is between $\frac 1{C|\gamma|}$ and $\frac{C}{|\gamma|}$.

By changing the preferred direction on the torus\footnote{technically,
  postcomposing the flat charts with a rotation matrix in
  $\mathrm{SL}_2(\mathbb{R})$ close to the identity.} we may assume that this rotation (when
rescaling the transversal to have length $1$) is in fact $B$--badly
approximable by the density observation above.
 
These flows are now $\frac{B}C$-badly approximable by the Sublemma.
 
It remains to modify the points. Given $q_1,...,q_k$ we choose $p_1,...,p_k$ that are the first times the vertical flows from the $q_i$ intersect our transversal. By Corollary \ref{cor:schmidt} we may choose $p_1',...,p_k'$ in a $\delta$ neighborhood of these points and on the transversal that are $(\frac{1}{4\cdot 4 \cdot 4})^{3(k-1)}$-badly approximable for the rotation (thought of as being on $[0,1)$). Applying the vertical flow (which is minimal) in the backwards direction, we can obtain $q_1',..,q_k'$, in a $\delta$ neighborhood for $q_1,...,q_k$, which are pairwise $c$-badly approximable for the flow for any $c<(\frac{1}{4\cdot 4 \cdot 4})^{3(k-1)} \frac 1 C$ (by our choice of transversal).
\end{proof}
Finally, we need the following density statement for
$(\epsilon,B)$--torus good $\omega$:
\begin{prop}\label{prop:density-torus} There exists $\epsilon>0$ so that for any $q_1,...,q_k\in S$ there exists $B>0$ so that 
  the set of $(\epsilon,B)$--torus good $\omega$ with respect to $q_1,...,q_k$
  is dense in $\widetilde{\Omega}(S)$. 
\end{prop}
\begin{proof}
  First note that the set of all $\omega$ which are lifts of Abelian
  differentials on tori branched over one point are dense in the space
  $\widetilde{\Omega}(S)$. Namely, this notion is invariant under
  the action of $\mathrm{Diff}(S, \{q_i\})$, and the desired density is true for 
  strata of Abelian differentials in the Hodge bundle over Teichm\"uller space.

  The desired density now follows from
  Proposition~\ref{prop:dense-ba}, since being torus good is invariant
  under pullback by differentials: if $\omega$ is torus good, and $\phi$ is 
  a diffeomorphism, then $\phi^*\omega$ is also torus good.
\end{proof}

\subsection{Point-pushing and torus good differentials}
Next, we describe constructions which allows us to modify a given
$(\epsilon,B)$--torus good $\omega$ in a simple way. In its
description, we think of simple closed curves as actual smooth, nonsingular maps from
$S^1 = \mathbb{R}/\mathbb{Z}$ to $S$, and not their isotopy classes. 
\begin{defin}
  We say that a simple closed curve $\alpha$ on $S$ is \emph{clean} for
  $\omega \in \widetilde{\Omega}(S)$ if
  \begin{enumerate}
  \item $\alpha$ is disjoint from all zeroes of $\omega$.
  \item $\alpha$ is transverse to the horizontal and vertical
    foliation of $\omega$.
  \end{enumerate}
\end{defin}
Observe that if $\alpha$ is clean, it intersects every horizontal or vertical segment (in the metric given by  $\omega$) in at most finitely many points, since the angle to the horizontal or vertical direction is bounded away from zero on the compact curve $\alpha$.

\begin{lem}\label{lem:clean-in-nbhd}
  For every $\omega \in \widetilde{\Omega}(S)$ there is a clean $\alpha$. Given any clean $\alpha$
  there is an open neighbourhood $U_{\omega,\alpha}$ of
  $\omega \in \widetilde{\Omega}(S)$, so that $\alpha$ is clean for
  every $\eta \in U_{\omega,\alpha}$.
\end{lem}
\begin{proof}
  This follows from Lemmas~4.2~and~4.7 of \cite{LS-connectivity}.
\end{proof}
\begin{defin}
  Suppose that $\alpha$ is a differentiable, simple closed curve on
  $S$ which is clean for some $\omega \in \widetilde{\Omega}(S)$. We
  say that $\alpha$ is parametrised with \emph{constant horizontal
    speed} if, in the flat charts defined by $\omega$, the horizontal
  deriviative $\alpha(t)$ is constant in $t$.
\end{defin}
Observe that any clean $\alpha$ admits a parametrisation with constant horizontal speed,
since by the definition of clean  $\alpha$ is nowhere vertical in the flat charts. 
\begin{prop}\label{prop:sliding-points}
  Suppose that $\omega$ is $(\epsilon,B)$--torus good with respect to
  marked points $q_1, \ldots, q_k$, and that $\alpha$ is a simple
  closed curve on $S$ with the following properties
  \begin{enumerate}
  \item $\alpha$ is clean for $\omega$, and parametrised with constant horizontal speed.
  \item There are $t_i$ so that $q_i=\alpha(t_i)$.
  \end{enumerate}
  Then for all $B'<B$ and any $s\in \mathbb{R}$, we have that 
  $\omega$ is $(\epsilon,B')$--torus good with respect to the
  marked points $\alpha(t_1+s), \ldots, \alpha(t_k+s)$.
\end{prop}
\begin{proof} 
  Put $q_i'=\alpha(t_i+s)$. Since the torus is a homogeneous space we
  may assume without loss of generality $q_1=q_1'$ and that therefore, by the choice of our
  parametrization, $q_j'$ is a translate along a vertical leaf from
  $q_j$ for all $j>1$; Let $\gamma'$ be a curve connecting $q_i'$ to $q_j'$ and
  $\gamma$ be the curve connecting $q_i=q_i'$ to $q_j$ by traversing
  $\gamma'$ and then the vertical segment of length $\ell$. Because
  $|g_u\gamma'|\geq |g_u\gamma|-e^{-\frac t 2 }\ell$ we have the the
  proposition.
\end{proof}

Next, we want to re-interpret the families of
$(\epsilon,B)$--torus good differentials constructed in
Proposition~\ref{prop:sliding-points} as paths in $\widetilde{\Omega}(S)$. It will be useful to describe this construction slightly more
generally.

\smallskip To begin, recall from
e.g.\cite[Section~4.2]{LS-connectivity} that associated to a simple
closed curve $\alpha$ there is an isotopy $D_{\alpha,t}:S\to S$ which
``pushes along the curve $\alpha$'', i.e. $D_{\alpha,t}(\alpha(s)) =
\alpha(t+s)$. 
Observe that such a diffeomorphism $D_{\alpha, t}$ preserves the
curve $\alpha$ setwise. 
Furthermore, note that any
diffeomorphism $F:S\to S$ defines by pullback a map $\widetilde{\Omega}(S) \to
\widetilde{\Omega}(S)$, which induces a map
\[ \widetilde{\Omega}(S) \to \widetilde{\Omega}(S) \] that
preserves geometric properties like being (eventually) strongly
$\epsilon$--thick, or having a vertical foliation with all leaves
closed.

Since $D_{\alpha,t}$ is a smoothly varying family of
diffeomorphisms, for any Abelian differential $\omega$, the assignment
\[ t \mapsto D_{\alpha,t}^{-1}\omega \] defines a continuous path
$C(\alpha, \omega)$ of Abelian differentials in
$\widetilde{\Omega}(S)$. Furthermore, this path depends continuously
on the initial differential $\omega$.

\smallskip As $\alpha$ is a closed curve, the endpoint
$D_{\alpha,1}^{-1}$ is actually a diffeomorphism fixing $(q_1,\ldots,
q_k)$. Hence, the endpoint of $C(\alpha, \omega)$ is obtained from the
initial point by pulling back the differential by that diffeomorphism.
This path in $\widetilde{\Omega}(S)$ depends on the choice of the
isotopy $D_{\alpha, t}$. 
Note that the mapping class of $S-\{q_1,\ldots, q_k\}$ defined by
$D_{\alpha,1}^{-1}$ depends only on the homotopy class of $\alpha$
relative to the set $\{q_1,\ldots, q_k\}$ and not the actual curve. We
call this mapping class a \emph{multi-point-push}, and denote it by
$P_\alpha$.  Observe that if $\alpha$ is embedded, then $P_\alpha$ is a
product of Dehn twists about curves to either side of $\alpha$. In
particular, results in Section~\ref{sec:train-track-north-south}
proved for (multi-)Dehn twists also apply for these $P_\alpha$.

We summarize some more basic properties of these paths in the following proposition.

\begin{prop}\label{prop:path-from-slide}
  Suppose that $\omega$ is an Abelian differential on $S$, and
  $\alpha$ is a clean simple closed curve which is parametrised with
  constant horizontal speed. Suppose that $q_i=\alpha(t_i)$ are points
  on the curve. Then there is a continuous path $c:[0,1] \to \Omega(S,
  \{q_1,\ldots, q_k\})$, whose endpoint $c(1)$ is the image of $c(0)$
  under the multi-point-push along $\alpha$. Furthermore, we have
  \begin{enumerate}
  \item If $\omega$ is $(\epsilon,B)$--torus good with respect to $q_1, \ldots, q_k$, then any
    point on $c$ $(\epsilon,B')$-thick for any $B'<B$ and thus it is eventually strongly $\delta$--thick with respect
    to $q_1,...,q_k$.
  \item If $\omega$ has vertical foliation a weighted multicurve, then
    the same is true for every point on $c$ (though the multi-curve may change).
  \item
    The path $c$ depends (for a fixed $\alpha$\footnote{Strictly, the curve is only fixed up to reparametrisation; for any $\omega$ one has to choose a constant horizontal speed parametrisation.}) continuously on
    the initial differential $\omega$.
  \end{enumerate}
\end{prop}
\begin{proof}
  \begin{enumerate}
  \item If we suppose that $\omega, q_1, \ldots, q_k$ and $\alpha$
    satisfy the requirements of Proposition~\ref{prop:sliding-points}
    then, for any $t$, the differential $D_{\alpha,t}^{-1}\omega$ is
    $(\epsilon,B)$--torus good with respect to $(q_1,\ldots, q_k)$\footnote{The covers certifying torus goodness vary in $t$, by pullback under the $D_{\alpha,t}^{-1}$}, by Proposition~\ref{prop:sliding-points}.
  \item If the vertical foliation of $\omega$ is a multicurve $C$, and $\phi$ is any
  diffeomorphism, then the vertical foliation of $\phi^*\omega$ is $\phi(C)$. In particular,
  the vertical foliation of $\phi^*\omega$ is also a multicurve. This implies the desired
  statement.
  \item This follows because the diffeomorphisms
    $D_{\alpha,t}^{-1}$ act continuously on $\widetilde{\Omega}(S)$, vary smoothly in $t$, and the map
    $\widetilde{\Omega}(S)\to\Omega(S, \{q_1,\ldots, q_k\})$ is continuous.
  \end{enumerate}
\end{proof}

\begin{lem}\label{lem:curve-from-push}
  Suppose that $q_i, \alpha$ are as in
  Proposition~\ref{prop:path-from-slide}, and $\omega$ is an
  Abelian differential whose vertical foliation is a multicurve
  $\delta$.  Consider the path in $\Omega(S, \{q_1,\ldots, q_k\})$
  from Proposition~\ref{prop:path-from-slide}. Then, only a finite
  number of weighted multicurves appear along this path as vertical
  foliations.
\end{lem}
\begin{proof}
  The fact that every vertical foliation along the path is a
  (weighted) multicurve follows from Proposition~\ref{prop:path-from-slide}. 
  Pick regular leaves $\gamma_1, \ldots,
  \gamma_n$, so that the vertical foliation of $\omega$ consists
  exactly of cylinders around the $\gamma_i$ (seen as a foliation on $S-\{q_1,\ldots, q_k\}$). Consider a time $s$ so that
  \[ D_{\alpha,s}^{-1}\left(\cup\gamma_i\right) \cap \{q_1,\ldots,q_k\} \neq \emptyset. \]
  Since $D_{\alpha, s}$ preserves $\alpha$, and all $q_i$ are contained in $\alpha$, for such an $s$ there has to be a point
  $\alpha(t_0) \in \gamma_i$ (for a suitable $i$) so that $\alpha(t_0+s) = q_j$ (for a suitable $j$). Hence, each such time $s$ corresponds to one of the finitely many intersection points
  of $\alpha$ with $\cup\gamma_i$ and a choice of $q_j$ -- in particular there are only finitely many such times, say $s_j,
  j=1, \ldots, J$. Observe that for $t \in (s_j,
  s_{j+1})$ the multicurves
  \[ t \to D_{\alpha,t}^{-1}\left(\cup\gamma_i\right), \quad t \in (s_j, s_{j+1}) \]
  are then all disjoint from the points $q_i$ by definition of the times $s_j$, and
  thus freely homotopic multicurves on the surface $S-\{q_1, \ldots, q_k\}$

  The multicurve defined by the vertical foliation of
  $D_{\alpha,t}^{-1}\omega$ is exactly
  $D_{\alpha,t}^{-1}\left(\cup\gamma_i\right)$, seen as a foliation on
  $S-\{q_1,\ldots, q_k\}$, and it is therefore constant (up to homotopy)
  for all $t \in (s_j, s_{j+1})$.

  This shows that the multicurve defined by the vertical foliation of
  $D_{\alpha,t}^{-1}\omega$ can only change at $t=s_j$ for some $j$,
  and therefore only takes finitely many values.
\end{proof}

\begin{defin}
  A \emph{twisting pair} for an Abelian differential $\omega$ and a
  number of marked points $q_1,\ldots, q_k$ is a pair of curves
  $\alpha, \beta$ so that
  \begin{enumerate}
  \item $\alpha, \beta$ fill $S$.
  \item $\alpha, \beta$ are clean.
  \item There are numbers $t_i, s_i$ so that $q_i=\alpha(t_i)=\beta(s_i)$.
  \end{enumerate}
\end{defin}
Arguing exactly as in the proof of \cite[Lemma~4.2]{LS-connectivity}, we see the following.
\begin{lem}\label{lem:existence-twisting}
  Let $S$ be a closed surface of genus at least $2$, and with some
  number of marked points $q_1,\ldots,q_k$.  For every $\omega \in
  \widetilde{\Omega}(S)$ there is a countable set of twisting pairs $(\alpha,
  \beta)$. Given any twisting pair $(\alpha, \beta)$ there is an open
  neighbourhood $U_{\omega,\alpha,\beta}$ of $\omega \in
  \widetilde{\Omega}(S)$, so that $(\alpha,\beta)$ is a twisting pair
  for every $\eta \in U_{\omega,\alpha,\beta}$.
\end{lem}
Indeed, \cite{LS-connectivity}'s argument involves a sequence of
simple closed curves converging to minimal foliations (which are
obtained by real parts of a rotated $\omega$), and they show that any
simple closed curves sufficiently close to the limits will have the
desired property -- in particular there will be countably many such
possible choices.

As an immediate consequence of Proposition~\ref{prop:path-from-slide}
and the fact that the product of multi-point-pushes around filling
curves are pseudo-Anosov, we have the following result, analogous to
\cite[Lemma~4.5]{LS-connectivity}.
\begin{cor}\label{cor:ppp}
  Let $(q_1, \ldots, q_k)$ be points on $S$. 
  Suppose that $\omega$ is an Abelian differential and 
  $\alpha,\beta$ is a twisting pair for $\omega$. Then, for any $j$, 
  define the diffeomorphism $\psi^{(j)}=P_\alpha^j(P_\alpha P_\beta^{-1})P_\alpha^{-j}$. 
  
  Let $F$ be the vertical foliation of $\omega$. Then, there is a \emph{point push path} in $\mathcal{PMF}$
  \[ P(F, \psi^{(j)} F) \] joining $F$ to $\psi^{(j)} F$. If $\omega$ is
  $(\epsilon,B)$--torus good with respect to $(q_1,\ldots,q_k)$, then any point on $P(F, \psi^{(j)} F)$ is a
  cobounded foliation. If $F$ is a multicurve, then there is a finite
  set of multicurves $F_i$ so that every point on $P(F, \psi^{(j)} F)$
  consists of a weighted multicurve homotopic to one of the $F_i$
  (with varying weights). 
\end{cor}
\begin{proof}
The idea is to apply Proposition~\ref{prop:path-from-slide} $2j+2$
  times to join $\omega$ to $\Psi^{(j)}\omega$ (where $\Psi^{(j)}$ is a
  diffeomorphism defining the multi-point-push $\psi^{(j)}$), and then obtain $P$
  as the associated path of vertical foliations.
  To make this precise, denote by $P_\alpha, P_\beta$ 
  point pushing diffeomorphisms around $\alpha, \beta$, and denote by
  $C(\eta, P_\ast\eta)$ the path of Abelian differentials guaranteed by applying
  Proposition~\ref{prop:path-from-slide}. We now form the concatenated path 
  \begin{equation*}
    \begin{aligned}
     C := C(\omega, P_\alpha \omega) \ast C(P_\alpha \omega,P_\alpha^{2} \omega) \ast
  \dots \\ C(P_\alpha^{j} \omega, P_\alpha^{j+1} \omega) \ast
  C(P_\alpha^{j+1} \omega, P_\alpha^{j+1}P_\beta^{-1} \omega) \ast\\
  C(P_\alpha^{j+1}P_\beta^{-1} \omega,
  P_\alpha^{j+1}P_\beta^{-1}P_\alpha^{-1} \omega) \ast \dots \\
  C(P_\alpha^{j+1}P_\beta^{-1}P_\alpha^{-j+1} \omega,
  P_\alpha^{j+1}P_\beta^{-1}P_\alpha^{-j} \omega). 
    \end{aligned}
  \end{equation*}
  Taking the vertical foliations, we then obtain a path in $\mathcal{PML}$
   \begin{equation}
    \begin{aligned}
  P := P(F, P_\alpha F) \ast P(P_\alpha F,P_\alpha^{2} F) \ast
  \dots \\ P(P_\alpha^{j} F, P_\alpha^{j+1} F) \ast
  P(P_\alpha^{j+1} F, P_\alpha^{j+1}P_\beta^{-1} F) \ast\label{eq:concatenated}\\
  P(P_\alpha^{j+1}P_\beta^{-1} F,
  P_\alpha^{j+1}P_\beta^{-1}P_\alpha^{-1} F) \ast \dots\\
  P(P_\alpha^{j+1}P_\beta^{-1}P_\alpha^{-j+1} F,
  P_\alpha^{j+1}P_\beta^{-1}P_\alpha^{-j} F)      
    \end{aligned}
    \end{equation}
  of foliations joining the vertical foliation $F$ of $\omega$ to $\psi^{(j)}(F)$.
  Proposition~\ref{prop:path-from-slide} says that if $\omega$ was
  $(\epsilon,B)$--torus good, the same is true for any point on the
  path $C$, hence $P$ consists of cobounded foliations. If $F$ was a
  multicurve, the claim follows from Lemma~\ref{lem:curve-from-push}.
\end{proof}

\begin{cor}\label{cor:peak}
  Suppose that $q_i, \omega, \alpha, \beta$ and $\psi^{(j)}$ are as in Corollary~\ref{cor:ppp},
  and suppose that $\omega$ is $(\epsilon, B)$--torus good.
  Then the concatenation
  \[ P(F, \psi^{(j)} F) \ast \psi^{(j)} P(F, \psi^{(j)} F) \ast \left(\psi^{(j)}\right)^2P(F, \psi^{(j)} F) \ast \dots \]
  extends to a continuous path of cobounded foliations connecting $F$ to the
  stable foliation of $\psi^{(j)}$.
\end{cor}
\begin{proof}
  First observe that $P(F, \psi^{(j)} F)$ is disjoint from the unstable
  foliation of $\psi^{(j)}$ for all $j$. Namely, the unstable foliation of
  $\psi^{(j)}$ has an angle-$\pi$ singularity, since it is a point-pushing
  map (compare \cite[Lemma~2.2]{LS-connectivity}), whereas $F$ (and any point
  push of it) as a lift of a foliation on a torus has no such
  singularities. Now, the corollary is an immediate consequence of the
  fact that pseudo-Anosov maps act on $\PMF$ with north-south
  dynamics.
\end{proof}

\begin{prop}\label{prop:skeleton-segment}
  Let $(q_1, \ldots, q_k)$ be
  points on $S$.  Suppose that $\omega$ is an Abelian differential
  whose vertical foliation is a multicurve, and let $\alpha,\beta$ be
  a twisting pair for $\omega$. Then, for any $j\in \mathbb{Z}$, define the mapping class
  $\psi^{(j)}=P_\alpha^j(P_\alpha P_\beta^{-1})P_\alpha^{-j}$.

  \begin{enumerate}
  \item There is a constant $C=C(\omega,\alpha,\beta)>0$, so that the union of the sets of multicurves
    appearing in paths $P(F, \psi^{(j)} F)$ from
    Corollary~\ref{cor:ppp} (over all $j$) has diameter at most $C$ in the curve graph. 
  \item If $\omega_n$ is a sequence of Abelian differentials
    converging to $\omega$, with vertical foliations $F_n$, then the
    paths $P(F_n, \psi^{(j)} F_n)$ converge to $P(F, \psi^{(j)} F)$.
  \end{enumerate}
\end{prop}
\begin{proof}
  \begin{enumerate}
  \item 
    We inductively consider the terms used in the proof of
    Equation~(\ref{eq:concatenated}) in Corollary~\ref{cor:ppp}. Let
    $\delta$ be one of the curves in the multicurve $F$.  By Lemma
    \ref{lem:curve-from-push} only finitely many curves appear in
    $P(F, P_\alpha F)$; call that set of curves $G_0$. In the next terms
  \[ P(P_\alpha^i F, P_\alpha^{i+1}F) = P_\alpha^i P(F, P_\alpha F) \]
  the curves that appear are the images of $G_0$ under powers of a
  Dehn multitwist, and as these act on the curve graph by isometries with fixed points
  (\emph{elliptically}), all curves that appear are contained in a
  $C_0$--neighbourhood of $\delta$. The curves appearing in the next
  two terms:
  \[ P(P_\alpha^{j} F, P_\alpha^{j+1} F) \ast P(P_\alpha^{j+1} F, P_\alpha^{j+1}P_\beta^{-1} F) = P_\alpha^{j} (P(F, P_\alpha F) \ast P(P_\alpha F, P_\alpha P_\beta^{-1} F))\]
  are images under $P_\alpha^j$ of the (finitely many) curves appearing
  in $P(F, P_\alpha F) \ast P(P_\alpha F, P_\alpha P_\beta^{-1} F)$, and are
  therefore also contained in some $C_1$--neighbourhood of $\delta$.
  Finally, in the terms of the third type
  \[ 
    P(P_\alpha^{j+1}P_\beta^{-1}P_\alpha^{-i} F, P_\alpha^{j+1}P_\beta^{-1}P_\alpha^{-i-1} F) = P_\alpha^{j+1}P_\beta^{-1}P_\alpha^{-i} P(F, P_\alpha^{-1} F) 
  \]
  we argue similarly. The path $P(F, P_\alpha^{-1} F)$ involves
  finitely many curves, which remain in a $C_3$--neighbourhood of
  $\delta$ by application of any power $P_\alpha^{-i}$. Let $C_4 =
  2C_3 + d(\delta, P_\alpha^{-1}P_\beta\delta)$. Then the image of the
  $C_3$--neighbourhood around $\delta$ is mapped by the pseudo-Anosov
  $P_\alpha^{-1}P_\beta$ into the $C_4$--neighbourhood around
  $\delta$.

 Finally, letting $C_5 = 2C_4 + d(\delta, \alpha)$, 
the $C_4$-neighbourhood around $\delta$ is sent to a
  $C_5$--neighbourhood around $\delta$ by the application of any further power of $P_\alpha^j$. Hence, $C_5$ has the desired property.
\item  This is a consequence of the fact that diffeomorphisms act
  continuously on the space of Abelian differentials; compare Proposition~\ref{prop:path-from-slide}~(3).
  \end{enumerate}
\end{proof}

\section{Paths in the punctured case, and Islands of point-pushes}
\label{sec:islands-point-push}
We now come to the main technical connectivity result for
$(\epsilon, B)$--torus good foliations. Let $S$ be a surface, and fix a finite
set of points $\mathbf{z}\neq\emptyset$.

Suppose $(\epsilon, B)$ are given so that there is a
$(\epsilon,B)$--torus good $\omega$ on $S$ with respect to
$\mathbf{z}$. We begin by defining $\mathcal{TG}$ to be
the set of all vertical foliations on $S$ of all $\omega$ which are
$(\epsilon,B)$--torus good. Since being $(\epsilon,B)$--torus good is
invariant under diffeomorphisms preserving $\mathbf{z}$, $\mathcal{TG}$ is a
mapping-class-group invariant set, and thus dense in
$\mathcal{F}(S,\mathbf{z})$.

Note that we have quantified our base point here a little differently than we did in the introduction (and is often done) when we defined $\epsilon$-cobounded foliations. Here we choose a best possible base point and an $(\epsilon,B)$-torus good foliation will be eventually $\epsilon'$-cobounded for any $\epsilon'<\epsilon$ and in particular it is in $\Cob$.

We begin with the following lemma,  whose proof is similar to the proof of Theorem~1.1 for $\mathbf{z}\neq\emptyset$ in \cite[Section~4.4]{LS-connectivity}.
\begin{lem}\label{lem:island-1}
  Suppose that $F, F' \in \mathcal{TG}$ are arbitrary. Then, there are
  \begin{enumerate}
  \item A finite number of simple closed curves
    $\alpha_i, \beta_i, i=1,\ldots, N-1$,
  \item numbers $\epsilon, B>0$,
  \item Abelian differentials $\omega_1, \ldots, \omega_N$,
  \item Abelian differentials $\hat{\omega}_1, \ldots, \hat{\omega}_N$,
  \item and multicurves $\delta_1, \ldots, \delta_N$, 
  \end{enumerate}
  so that the following hold:
  \begin{enumerate}[i)]
  \item The vertical foliation of $\omega_1$ is $F$, and the vertical
    foliation of $\omega_N$ is $F'$.
  \item The curves $(\alpha_i, \beta_i)$ are a twisting pair for both
    $\omega_i, \hat{\omega}_i$ and $\omega_{i+1}, \hat{\omega}_{i+1}$.
  \item All $\omega_i$ are $(B,\epsilon)$--torus good, given by
    pullbacks of $\omega^i_T$ along a cover $p_i:S\to T$,
  \item all $\hat{\omega}_i$ are pullbacks of Abelian differentials
    $\hat{\omega}^i_T$ whose vertical foliation is a single cylinder.
  \item the multicurves $\delta_i$ are the core curves of the
    vertical cylinders of the $\hat{\omega}_i$.
  \end{enumerate}
  Moreover, we can choose a countable family as above that only overlap as necessary. That is, for each such sequence the first of the $\omega_i$ has vertical foliation $F$ and the last has vertical foliation $F'$, but all other chosen vertical foliations, differentials and curves are distinct.
\end{lem}
\begin{proof} We first produce a single collection as in (1)--(3) that satisfy i)-iii). 
  Let $F, F' \in \mathcal{TG}$ be given. By definition of
  $\mathcal{TG}$, they are vertical foliations of
  $(\epsilon, B)$--torus good Abelian differentials
  $\omega_1,\omega_N$. Note that these satisfy i) and iii) by choice.
  Choose a path $\gamma(t)$ from $\omega_1$ to $\omega_N$ in
  $\widetilde{\Omega}(S)$. For every $\gamma(t)$ there is a twisting
  pair $\alpha(t), \beta(t)$ for $\gamma(t)$ by
  Lemma~\ref{lem:existence-twisting}. In fact, by the same lemma,
  there is a small open neighbourhood $U_t$ so that the curves
  $\alpha(t), \beta(t)$ are twisting pairs for all differentials in
  $U_t$. By compactness of the path $\gamma$, a finite number
  $U_1, \ldots, U_N$ of such neighborhoods suffice to cover the path
  $\gamma$. We let $\omega_i \subset U_i \cap U_{i-1}$ be
  $(\epsilon, B)$--torus good (which is possible since
  $(\epsilon, B)$--torus good differentials are dense), and
  $p_i:S\to T$ the defining covering.  This implies the existence of
  the desired objects (1) through (3) with properties i) through iii). 
  
  We now inductively produce a countable collection as in (1)--(3) that satisfy i)-iii) and additionally only overlap as necessary. 
  . Note that because Lemma \ref{lem:existence-twisting} produces a countable family of curves, if we are given 
  $$\Big\{\{U_i^{(j)},(\alpha_i^{(j)},\beta_i^{(j)}),\omega_i^{(j)}\}_{i=1}^{N_j}\Big\}_{j=1}^k$$ we can produce a new sequence $U_1^{(k+1)},...U_{N_{j+1}}^{(k+1)}$, $(\alpha_1^{(k+1)},\beta_{1}^{(k+1)}),...,(\alpha_{N_{k+1}}^{(k+1)},\beta_{N_{k+1}}^{(k+1)})$
 and  $\omega_1^{(k+1)},...,\omega_{N_{k+1}}^{(k+1)}$ so that $\omega_i^{(k+1)}$ does not appear in the above list for any $i\notin \{1,N_{k+1}\}$ and neither $\alpha_i^{(k+1)}$ nor $\beta_i^{(k+1)}$ appear in the previous list for any $i$. Indeed for each $\omega \in \gamma(t)$ there are a countable choice of $\alpha$ and $\beta$ and torus good differentials are dense.

  It remains to show the existence of Abelian differentials and curves
  as in (4),(5) with properties iv) and v). Let $\omega_T^i$ be the
  differential so that $\omega_i$ is the pullback of $\omega_T^i$ by
  $p_i$. As the
  $U_i$ are open, and cylinder directions are dense, there is a
  differential $\hat{\omega}_T^i$ on $T$ which has vertical direction
  a simple closed curve $\delta_i'\subset T$, and so that the pullback
  $p_i^*\hat{\omega}_T^i = \hat{\omega}_i$ is also contained in
  $U_i\cap U_{i-1}$. By the definition of that set, $(\alpha_i, \beta_i)$ and
  $(\alpha_{i-1}, \beta_{i-1})$ are then 
  twisting pairs for $\hat{\omega}_i$, so they satisfy
  iii). Furthermore, the vertical foliation of $\hat{\omega}_i$ is the
  multicurve $p_i^{-1}(\delta_i') = \delta_i$. Hence,
  $\hat{\omega}_i, \delta_i$ satisfy properties iv) and v). We can run this argument for each of the $\{U_i^{(j)},(\alpha_i^{(j)},\beta_i^{(j)}),\omega_i^{(j)}\}_{i=1}^{N_j}$ produced above and  using the density of single cylinder surfaces, we can ensure 
     the $\hat{\omega}_i^{(j)}$ only overlap as necessary. 
\end{proof}

Using the output of Lemma~\ref{lem:island-1}, we can construct paths
between torus good foliations in the following way.
\begin{defin}\label{def:peak-and-path}
  Let $\omega_i, (\alpha_i,\beta_i), p_i, \omega_T^i$ be as in Lemma~\ref{lem:island-1}.
For each $i$ choose a number $K_i$ and a corresponding mapping class
\[ \psi^{(K_i)}_i = P_{\alpha_i}^{K_i} (P_{\alpha_i}P_{\beta_i}^{-1})
P_{\alpha_i}^{-K_i} \] which we call the \emph{peak pseudo-Anosovs},
and for each $i=2, \ldots, N-1$ choose a cobounded foliation $F_i$
which is a lift of a foliation under $p_i$, which we call the
\emph{base foliations}, so that the $(\alpha_i, \beta_i)$ are a
twisting pair for that lift. Put $F= F_1, F'= F_N$.
  
  \smallskip The associated \emph{push-and-peak-path} is then the path
  $\gamma$ obtained as a concatenation
  \[ \gamma = \gamma^+_1 \ast \overline{\gamma_2^-} \ast \gamma_2^+\ast \dots \ast
     \gamma_{N-1}^+ \ast \overline{\gamma^-_N}, \] 
    where $\overline{\cdot}$
    denotes the path with opposite orientation, 
  in the following way:
  \begin{enumerate}
  \item $\gamma^+_i$ is the path starting in $F_i$, and ending in the
    stable foliation of $\psi^{(K_i)}_i$ which is obtained as the concatenation
    \[ P(F_i, \psi^{(K_i)}_i F_i) \ast \psi^{(K_i)}_i P(F_i, \psi^{(K_i)}_i F_i) \ast \left(\psi^{(K_i)}_i\right)^2
    P(F_i, \psi^{(K_i)}_i F_i) \ast \ldots \] of images of the point-push path
    $P(F_i, \psi^{(K_i)}_i F_i)$ (compare Corollary~\ref{cor:ppp} and~\ref{cor:peak}) under
    $\psi^{(K_i)}_i$.
  \item $\gamma^-_i$ is the path starting in $F_i$, and ending in the
    stable foliation of $\psi^{(K_{i-1})}_{i-1}$ which is similarly obtained as the concatenation
      \[ P(F_i, \psi^{(K_{i-1})}_{i-1} F_i) \ast \psi^{(K_{i-1})}_{i-1}
      P(F_i, \psi^{(K_{i-1})}_{i-1} F_i) \ast \left(\psi^{(K_{i-1})}_{i-1}\right)^2
      P(F_i, \psi^{(K_{i-1})}_{i-1} F_i) \ast \cdots  \]  
  \end{enumerate}
\end{defin}
Observe that peak-and-push paths are defined using Abelian
differentials, but really depend only on the choice of suitable
$F_1,...,F_N$; $\alpha_1,\beta_1,...,\alpha_N,\beta_N$ and
$K_1,...,K_N$.
\begin{cor}\label{cor:tg-connected}
  Any two points in $\mathcal{TG} \subset \PMF$ can be joined by countably many 
  paths of cobounded foliations which only intersect at the first and last foliations.
\end{cor}
\begin{proof}
  Let $F, F' \in \mathcal{TG}$ be given. Apply
  Lemma~\ref{lem:island-1} to obtain $\omega_i, (\alpha_i, \beta_i),
  p_i, \omega_i^T$. Construct the push-and-peak path as in
  Definition~\ref{def:peak-and-path}. First observe that by using
  Corollary~\ref{cor:peak}, we see that this is indeed a continous
  path of cobounded foliations. It joins $F$ to $F'$ by construction. To obtain disjoint paths, we apply Lemma \ref{lem:island-1} to obtain a sequence of different $\omega_i^{(j)}$, and we further stipulate that $\omega_i^{(j)}$ and $\omega_{i'}^{(j')}$ project to different foliations on the torus if $j\neq j'$. This ensures that the the path segments, $\gamma_i^{\pm}$ for different $j$ do not overlap except possibly at the stable foliations of the psuedo-Anosovs. As $K_i$ goes to infinity the stable foliations of the $\psi_i^{(K_i)}$ converge to $\alpha_i$. Given these paths for 
  $\{\{\omega_i^{(j)},(\alpha_i^{(j)},\beta_i^{(j)})\}_{i=1}^{N_j}\}_{j=1}^r$, by choosing $K_i$ large enough given these we can ensure the path we build from 
  $\{\omega_i^{(r+1)},(\alpha_i^{(r+1)},\beta_i^{(r+1)})\}_{i=1}^{N_{r+1}}$ is disjoint from the previous $r$ paths.
\end{proof}

We now use the machinery developed in
Section~\ref{subsec:curve-graph-machine} in order to contract suitable point-push-paths into small neighbourhoods of uniquely ergodic foliations.
We briefly recall the setup from that section. Namely, suppose that $(\tau_n)$ is a full
splitting sequence in the direction of a uniquely ergodic foliation
$E$, and let $(f_m,k_m)$ be an associated $\Mcg$-sequence. 

Recall from Section~\ref{sec:train-track-north-south} (in particular, the discussion around Lemma~\ref{lem:split-means-finite}) that there are
(nested) neighbourhoods $U_n(E, \tau)$, so that
\[\bigcap_{n}U_n(\tau, E) = \{ E \} \]
and finitely many ``model neighbourhoods'' $\mathcal{U}^{(k)}$, so that
\[ f_n\left(\mathcal{U}^{(k(n))}\right) = U_n(\tau, E). \] 

The following theorem is concerned with finding paths $P$ which connect
two points in a model neighbourhood $\mathcal{U}^{(k)}$, and which are also moved
by the $f_n$ into smaller and smaller neighbourhoods of $E$ (even though 
the path $P$ may leave $\mathcal{U}^{(k)}$!).

\begin{thm}\label{thm:contracted-paths}
  Suppose that $(\tau_n)$ is
  a full splitting sequence in the direction of a uniquely ergodic
  foliation $E$, and let $(f_m,k_m)$ be an associated $\Mcg$-sequence.

  Fix an essential type $k$, and let 
  $F, F' \in \mathcal{TG} \cap \mathcal{U}^{(k)}$ be two foliations defined
  by torus good Abelian differentials $\omega, \omega'$. Furthermore let $\delta, \delta'$ be
  lifts of simple closed curves on the base tori.
  Assume that
  \begin{description}
  \item[$(\ast)$] $\mathcal{U}^{(k)}$ contains every foliation which is a
    lift of the torus covers defined by $\omega, \omega'$.
  \end{description}
  Then for any $n$ and $r$ there is a number $m_0$ with the following
  property.  For any $m>m_0$ with $k_m=k$ there are 
  peak-and-push paths
  $\gamma_1,...,\gamma_r$ connecting $F$ to $F'$, which intersect only at $F,F'$ and so that $f_m\gamma_i$ is completely
  contained in $U_n(\tau, E)$.

  Without property $(\ast)$ the conclusion remains true for $F, F'$ which are
  sufficiently close (depending on $m$)
   to the curves $\delta, \delta'$.
\end{thm}
\begin{proof}
  We begin by noting that due to property $(\ast)$, the initial
  segment $\gamma_1^+$ and terminal segment $\gamma_N^-$ are
  automatically contained in $\mathcal{U}^{(k)}$, independent of all
  other choices.  Hence, for any $m>n$ with $k_m=k$, the images of
  these segments under $f_m$ are contained in $U_n(\tau, F)$, by
  Equation~(\ref{eq:relation-u-up-down}) of the
  associated sequence. If $(\ast)$ does not hold, we will argue for
  the initial/terminal segment exactly as below.

  \smallskip We will now explain how to construct the path segments
  $\gamma_i^+$ of the push-and-peak-path; the segments $\gamma_i^-$
  will be constructed analogously. Whenever a constant $K_i$ is
  chosen, it needs to be chosen to be large enough for the
  construction of both $\gamma_i^+$ and $\gamma_{i+1}^-$.
  
  \smallskip 
  \textit{Definition \ref{def:peak-and-path} produces a set of bounded diameter in $\mathcal{C}(S)$:} Let $\omega_i, (\alpha_i, \beta_i), \delta_i$ be the
  objects guaranteed by Lemma~\ref{lem:island-1} applied to $F, F'$.
  Consider the point-pushing pseudo-Anosov map
  \[ \psi_i^{(K_i)} =
  P_{\alpha_i}^{K_i}P_{\alpha_i}P_{\beta_i}^{-1}P_{\alpha_i}^{-K_i}.\]
  By Proposition~\ref{prop:skeleton-segment},  there are numbers $C_i$, depending only on $\alpha_i,\beta_i$ and $\delta_i$
  so that for any $L \in \mathbb{Z}$ 
  every point on the peak-and-push paths $P(\delta_i,
  \psi^{(L)}_i\delta_i)$ corresponds to a multicurve which is
  contained in the $C_i$--neighbourhood of $\delta_i$ in the curve
  graph. Let $G_i$ be the set of all
  multicurves appearing on such paths.  Because $\delta_1,...,\delta_N$ are given and each $G_i$ is 
  	contained in a $C_i$--neighbourhood of $\delta_i$ in the curve graph, the (finite) union
  \[ G = \bigcup_{i=1}^N G_i \]  also has
  finite diameter in the curve graph.  We can therefore choose a number $d$ large enough so that
  for all $i$, every curve in $G$ has distance at most $d$ from
  $\alpha_i$. 

\textit{Obtaining $m$ from Proposition~\ref{prop:apply-pseudo}:}
  By increasing $d$, we may also assume that (for any choice of powers
  $K_i$ in the push-and-peak-paths), the quasi-axes of $\psi^{(K_i)}_i$ pass
  within distance $d$ of $\alpha_i$ as well.
  Indeed, if $\rho$ is a quasi-axis for $\psi_i^{(0)}$ then $P_{\alpha_i}^{K_i}\rho$ is a quasi axis for $\psi_i^{(K_i)}$, and the claim follows since $P_{\alpha_i}$ fixes $\alpha_i$.
  
  \smallskip Apply Proposition~\ref{prop:apply-pseudo} with this $d$ to
  $P_{\alpha_i}P^{-1}_{\beta_i}$ as the pseudo-Anosov, and
  $\mathcal{V} = U_n(\tau, F)$ as the neighbourhood and any
  curve in $G$ as the curve $\beta_0$ for every $i$ to get a constant $N=N_i$.
  Let $m_0$ be the maximum of these constants.

\textit{Choosing $K_i$ large enough to obtain contraction:}
  Let now $m>m_0$ be given. Then Proposition~\ref{prop:apply-pseudo} yields\footnote{noting that since multi-point pushing maps are multitwists, we can apply that Proposition in this situation}
  that if we choose the powers $K_i$ in the definition of
  $\psi^{(K_i)}_i$ large enough, the images of the point-pushing paths
  $\left(\psi^{(K_i)}_i\right)^jP(\delta_i, \psi^{(K_i)}_i\delta_i), \left(\psi^{(K_{i-1})}_{i-1}\right)^jP(\delta_i,
  \psi^{(K_{i-1})}_{i-1}\delta_i)$ under $f_m$ are contained in
  $U_n(\tau, F)$ for all $j$.
  
  As we let $K_i \to \infty$, the stable foliation of $\psi^{(K_i)}_i$
  converges to $\alpha_i$. Hence, we can choose numbers $K_i$ large
  enough, so that the stable foliation of $\psi_i^{(K_i)}$ is
  sent into $U_n(\tau, E)$ by $f_m$ (in addition to the previous
  constraints).
  
  Since the pseudo-Anosov $\psi^{(K_i)}_i$ acts on $\PMF$ with
  north-south-dynamics, and the (compact) path
  $P(\delta_i, \psi^{(K_i)}_i\delta_i)$ does not intersect the unstable
  foliation of $\psi^{(K_i)}_i$ (as the path consists only of multicurves),
  there is a number $\epsilon>0$ and $J>0$ so that the
  $\epsilon$--neighbourhood of $P(\delta_i, \psi^{(K_i)}_i\delta_i)$ is mapped
  into $U_n(\tau, E)$ by $f_m\left(\psi^{(K_i)}_i\right)^j$ for all $j > J$.

  By the continuity of the maps $f_m\psi^{(K_i)}_i,...,f_m \left(\psi^{(K_i)}_i\right)^{J}$, we may
  therefore choose $F_i$ close enough to $\delta_i$ so that in fact the
  path $P(F_i, \psi^{(K_i)}_i F_i)$ is contained in the $\epsilon$--neighbourhood of 
  $P(\delta_i, \psi^{(K_i)}_i\delta_i)$, and therefore 
  \[ f_m\left(\psi^{(K_i)}_i\right)^jP(F_i, \psi^{(K_i)}_i F_i) \]  is contained in $U_n(\tau,
  E)$ for all $j>J$. 

  By the continuity of the maps $\psi^{(K_i)}_i,\left(\psi^{(K_i)}_i\right)^2,...,\left(\psi^{(K_i)}_i\right)^J$ we can
  choose the foliation $F_i$ even closer to $\delta_i$, to ensure that
  for all $j\geq 0$, since the paths $f_m\left(\psi^{(K_i)}_i\right)^jP(\delta_i,
  \psi^{(K_i)}_i\delta_i)$ are all contained in $U_n(\tau,
  F)$. Repeating the same argument for all $i$, and analogously
  for the paths for $\gamma_i^-$ finishes the argument.
  
  \textit{A finite number of curves:} To obtain the result for curves
  $\gamma_1,...,\gamma_r$ we use Lemma \ref{lem:island-1} and
  Corollary \ref{cor:tg-connected} to produce $\{\omega_i^{(k)},
  (\alpha_i^{(k)}, \beta_i^{(k)}), \delta_i^{(k)}\}_{k=1}^r$. We then
  apply the step, \textit{Definition \ref{def:peak-and-path} produces
    a set of bounded diameter in $\mathcal{C}(S)$}, to each
  $\{\omega_i^{(k)}, (\alpha_i^{(k)}, \beta_i^{(k)}),
  \delta_i^{(k)}\}$ to produce $G^{(k)}$. Let $\tilde{G}=\cup_{k=1}^r
  G^{(k)}$ and observe that it is a set of bounded diameter in
  $\mathcal{C}(S)$. So there exists $\tilde{d}$ so that $\tilde{G}$ is
  in a $\tilde{d}$ neighborhood of $\alpha_i^{(j)}$ for all $i,j$. We
  may apply Proposition \ref{prop:apply-pseudo} to obtain a uniform
  $m$ for all possible $P_{\alpha_i^{(k)}}P_{\beta_i^{(k)}}$. Once we
  have this $m$ we can treat each $\gamma_k$ separately in the step
  \textit{Choosing $K_i$ large enough to obtain contraction}. We can
  arrange the disjointness of the paths as in Corollary
  \ref{cor:tg-connected}.
\end{proof}
The following corollary will be used in the next section 
to build paths of foliations on surfaces without punctures.
\begin{cor}\label{cor:passing-to-branched-covers}
  Suppose that $S$ is a surface with marked points, which is a branched cover over a torus. Suppose further that $\hat{S}$ is a closed surface and $p:\hat{S} \to S$ is a properly branched cover, with
  branching set $\mathbf{z}$ equal to the marked points of $S$. Suppose that $\tau_n$ is a splitting
  sequence of train tracks on $\hat{S}$ in the direction of a uniquely
  ergodic foliation $\hat{E}$. Let $f_1, \ldots$ be an associated $\Mcg$-sequence.

  Fix an essential type $k$, and let
  $\hat{F}=p^{-1}(F), \hat{F}'=p^{-1}(F') \in \mathcal{U}^{(k)}$ be lifts under $p$
  of torus good foliations $F, F'$ on $(S,\mathbf{z})$ , defined by
  Abelian differentials $\omega, \omega'$, and let $\delta, \delta'$ be lifts of
  simple closed curves on the base tori. Assume
  that 
  \begin{description}
  \item[$(\ast)$] $\mathcal{U}^{(k)}$ contains every lift under $p$ of
    a foliation on $S$ which is a lift of the torus covers defined by
    $\omega, \omega'$.
  \end{description}
  Then for any $n$ and $r$ there is a number $m_0$ with the following
  property. For any $m>m_0$ with $k_m=k$ there are 
   peak-and-push paths
  $\gamma_1,...,\gamma_r$ connecting $F$ to $F'$, which lifts under $p$ to a paths
  $\hat{\gamma}_1,...,\hat{\gamma}_r$ of cobounded foliations, and so that each $f_m\hat{\gamma}_i$
  is completely contained in $U_n(\tau, F)$. Moreover the $\hat{\gamma}_i$ intersect only at $f_mF$, $f_mF'$.

  Without property $(\ast)$ the conclusion remains true for $\hat{F}, \hat{F}'$ which are
  sufficiently close (depending on $\tau$) to lifts $\hat{\delta}, \hat{\delta'}$ of the curves $\delta, \delta'$.
\end{cor}
\begin{proof}
  This follows exactly like the previous proof, using that the lifting
  map $\PMF(S) \to \PMF(\hat{S})$ is a continuous embedding.
\end{proof}

\subsection{Proof of the main theorem in the punctured case}
As an application of Theorem~\ref{thm:contracted-paths}, we can now
prove the main theorem in the case of punctured surfaces.
\begin{thm}\label{thm:main-punctured}
  Suppose that $\Sigma$ is a surface of genus $g \geq 2$ and with
  $p \geq 1$ punctures.  Then the set of uniquely ergodic foliations
  on $\Sigma$, $\mathcal{UE}(\Sigma)$, is path-connected. Moreover, given any finite set $S$, we have that $\mathcal{UE}(\Sigma) \setminus S$ is path connected. 
\end{thm}
To prove the theorem, the main step is to show that one can connect an
arbitrary uniquely ergodic foliation $E$ to a torus good
foliation. By Corollary~\ref{cor:tg-connected} this will be enough to show 
the theorem.

Before beginning the proof in earnest, let us remark quickly about the case that $\Sigma$ has only a 
	single marked point. In the previous section, we usually thought of the surfaces to have at least two marked points
	(since this is much harder, and the naturally occuring case in the proof of the main theorem in the closed case).
	However, for a single marked point we simply set the ``pairwise badly approximable'' condition in the definition of
	torus good differentials to be empty. A torus good differential in this sense has cobounded
	vertical foliation, and (any) point-push (of the single point) preserves this property. Hence, the desired results
	also hold in this case.
	
In order to prove the main step, we use the connection to splitting sequences
described in Section~\ref{sec:train-track-north-south}.

To this end, let $\tau$ be a maximal train track carrying $E$,
and $\tau_s$ a full splitting sequence in direction of $E$. We
let $(f_n, k_n)$ be an associated $\Mcg$-sequence. First, we need the
following statement, purely about the model neighbourhoods.
\begin{lem}\label{lem:initial-island}
  Given any $k$ there is a torus cover $p_k:\Sigma \to T$, so that the lift of
  every foliation from $T$ via $p_k$ is contained in $\mathcal{U}^{(k)}$.
\end{lem}
\begin{proof}
  Let $p:\Sigma\to T$ be any branched torus cover, and let $L\subset \PMF$
  be the set of all lifts of foliations on $T$ via $p$. Precomposing
  the cover $p$ by a mapping class $\varphi^{-1}$ replaces $L$ by $\varphi(L)$.

  Choose a pseudo-Anosov $\varphi$ whose attracting foliation is
  contained in the (open) set $\mathcal{U}^{(k)}$, and whose repelling
  foliation is not contained in $L$. As
  pseudo-Anosovs act on $\PMF$ with north-south dynamics, there
  is a power $N$ so that $\varphi^N(L) \subset \mathcal{U}^{(k)}$, which
  shows the existence of the desired cover.
\end{proof}
From now on, we fix for each $k$ covers $p_k$ as in
Lemma~\ref{lem:initial-island}.  Furthermore we choose, once and for
all, torus good foliations $F^{(k)} \in \mathcal{U}^{(k)} \setminus
\Mcg(\Sigma)S$ which are defined by these covers $p_k$. Note that
these exist, since there are uncountably many torus-good foliations.

Recall that the associated $\Mcg$-sequence has the property that 
\[ U_s(\tau, E) = f_s(\mathcal{U}^{(k_s)}).\]
In particular, the (torus good) foliations $f_s\left(F^{(k_s)}\right)$
converge to $E$. Our strategy will be to find paths $\gamma_s$ of cobounded foliations which connect $f_s\left(F^{(k_s)}\right)$
to $f_{s+1}\left(F^{(k_{s+1})}\right)$, so that the concatenated paths
\[ c_n = \gamma_1\ast\gamma_2\ast\cdots\gamma_n \]
converge, as $n\to \infty$, to a path connecting the torus good
foliation $f_1\left(F^{(k_1)}\right)$ to $E$.

Recall from Lemma~\ref{lem:split-means-finite} that there is a finite
set $M$ of mapping classes, so that for all $n$ we have
\[ f_n^{-1}f_{n+1} \in M. \]
The following corollary of Theorem~\ref{thm:contracted-paths} is what
makes our construction of paths work:
\begin{cor}\label{cor:increment-paths}
  Given any $n$, and $r$ there is a number $m$ with the following property: if
  $s > m$, then there are 
  paths $\gamma_s^{(1)},...,\gamma_s^{ (r)}$ with the following properties:
  \begin{enumerate}
  \item $\gamma_s^{(j)}$ joins $f_s\left(F^{(k_s)}\right)$ to
    $f_{s+1}\left(F^{(k_{s+1})}\right)$ for all $1\leq j\leq r$,
  \item $\gamma_s^{(j)}$ consists only of cobounded foliations for all $1\leq j\leq r$, and
  \item $\gamma_s^{(j)} \subset U_n(\tau, E)$ for all $1\leq j\leq r$.
  \item For all $j,\neq j'$, $\gamma_s^{(j)}\cap \gamma_s^{(j')}=\{f_s\left(F^{(k_s)}\right),f_{s+1}\left(F^{(k_{s+1})}\right)\}$  for all $1\leq j\leq r$.
  \end{enumerate}
\end{cor}
\begin{proof}
  Since we only make a claim about large $s$, we may assume without
  loss of generality that every type $k_s$ for $s > m$ is essential.

  We will then find $\gamma_s$ as
  \[ \gamma_s = f_s\iota_s \]
  for a suitable path $\iota_s$. 
In order to satisfy (1), the path
  $\iota_s$ needs to join $F^{(k_s)}$ to
  $f_s^{-1}f_{s+1}\left(F^{(k_{s+1})}\right)$. Note that by the second
  claim of Lemma~\ref{lem:split-means-finite} (Equation~(\ref{eq:crossnesting})) we have
  \[ f_s^{-1}f_{s+1}\left(\mathcal{U}^{(k_{s+1})}\right) \subset 
    \mathcal{U}^{(k_{s})}\]
  and therefore we have that
  \[ F^{(k_s)}, f_s^{-1}f_{s+1}\left(F^{(k_{s+1})}\right) \in
    \mathcal{U}^{(k_{s})}. \] In fact, as the foliations $F^{(k_s)}$
  are defined by the covers from Lemma~\ref{lem:initial-island}, the
  foliations
  $F = F^{(k_s)}, F' = f_s^{-1}f_{s+1}\left(F^{(k_{s+1})}\right)$ are
  defined by Abelian differentials $\omega, \omega'$ which satisfy
  condition $(\ast)$ in Theorem~\ref{thm:contracted-paths} by the comment right after the proof of  Lemma~\ref{lem:initial-island}.

  Hence, for any essential type $k$ and $s$ with $k_s=k$, we can apply
  Theorem~\ref{thm:contracted-paths} to $r$, $U_n(\tau, E)$ and pairs
  of foliations
  $(F^{(k)}, f_s^{-1}f_{s+1}\left(F^{(k_{s+1})}\right))$, to obtain
  thresholds
  $m_0(k, F^{(k)}, f_s^{-1}f_{s+1}\left(F^{(k_{s+1})}\right))$. Note
  that since there are finitely many $F^{(i)}$ and for all $s$,
  $f_s^{-1}f_{s+1}\in M$ for the finite set $M$ from
  Lemma~\ref{lem:split-means-finite}, there is a number
  \[ m = \max m_0(k, F^{(k)},
    f_s^{-1}f_{s+1}\left(F^{(k_{s+1})}\right)). \] We claim that this
  has the desired property. Namely, suppose that $s > m$. Then, let
  $k=k_s$ be the type of the index $s$. By our choice of $m$ the
  foliations $F^{(k_s)}, f_s^{-1}f_{s+1}\left(F^{(k_{s+1})}\right)$
  and the number $s$ then satisfy the prerequisites of
  Theorem~\ref{thm:contracted-paths}, and we can choose $\iota_s$ to
  be the path guaranteed by that theorem. Since peak-and-push-paths
  consist only of cobounded foliations, and this property is invariant
  under the mapping class group, $f_s\iota_s$ then satisfies (1) and
  (2). Property (3) and (4) are  directly guaranteed by
  Theorem~\ref{thm:contracted-paths}. 
  
\end{proof}

\begin{proof}[{Proof of Theorem~\ref{thm:main-punctured}}]
  In order to show the theorem, in light of Corollary~\ref{cor:tg-connected}
  it suffices to show that, any
  uniquely ergodic foliation $E \notin S$ can be joined to a torus good
  foliation by a path that doesn't intersect $S$. We will do this by using the construction outlined above.

  Namely, apply Corollary~\ref{cor:increment-paths} for every $n$ to
  get a sequence $m_n$ of threshold indices. We may assume without
  loss of generality that $m_n$ is increasing in $n$.  For
  $s \leq m_1$, choose $\gamma_s$ to be any path of cobounded
  foliations connecting $f_s\left(F^{(k_s)}\right)$ to
  $f_{s+1}\left(F^{(k_{s+1})}\right)$ that does not intersect $f_{s}^{-1}S$. Indeed, we apply Corollary~\ref{cor:increment-paths} with $r=|S|+1$. 
 For $m_{n+1} \geq s > m_n$, let
  $\gamma_s$ be the result of applying
  Corollary~\ref{cor:increment-paths}. We then have that
  $\gamma_s \subset U_n(\tau, E)$ for $m_{n+1} \geq s > m_n$.

  Consider now the paths
  \[ c_r = \gamma_1 \ast \cdots \ast \gamma_r, \] and note that they
  join the torus good foliation $f_1\left(F^{(k_1)}\right)$ to
  $f_{r+1}\left(F^{(k_{r+1})}\right)$. For any $s < r$, let
  \[ i_{s,r} = \gamma_{s+1} \ast \cdots \ast \gamma_r, \]
  so that
  \[ c_r = c_s \ast i_{s,r}. \]
  By our construction of the $\gamma_s$, we have that for any $n$ there is
  some $m_n$, so that for all $r>s>m_n$:
  \[ i_{s,r} \subset U_n(\tau, E) \]
  As by Corollary~\ref{cor:u-nest} we have that
  \[ \bigcap_n U_n(\tau, E) = \{E\}, \]
  this shows that since $c_r\subset U_n$ for all $r>m_n$, the infinite concatenation
  \[ c_\infty = c_1\ast c_2\ast\cdots \ast c_n \ast\cdots \] 
  extends to a continuous path with
  endpoints $f_1\left(F^{(k_1)}\right), E$, finishing the proof.
\end{proof}

\section{Paths in the closed case, and Islands of branched covers}
\label{sec:closed-case}
\begin{thm}\label{thm:main-closed}
  Suppose that $\Sigma$ is a closed surface of genus $g \geq 5$.  Then
  the set $\mathcal{UE}(\Sigma)$ of uniquely ergodic foliations on $\Sigma$ is
  path-connected. Moreover, for any finite set $S$, $\mathcal{UE}(\Sigma)\setminus S$ is path connected.
\end{thm}
To prove this theorem, we want to run the strategy of the proof of Theorem~\ref{thm:main-punctured}, with the addition of using branched covers
to lift paths from punctured to closed surfaces.

The first ingredient is the following theorem, which follows
from the methods developed in \cite{LS-connectivity}. 
\begin{prop}\label{prop:ls-islands}
  Suppose that $g \geq 5$. Then there is an involution $\sigma$ of the
  closed surface $\Sigma_g$ with the following properties.
  \begin{enumerate}[i)]
  \item $\Sigma_g / \sigma$ is a surface of genus at least $2$ with several marked points.
  \item For any conjugate $\hat{\sigma}$ of $\sigma$ in the mapping class
    group there is a sequence $\sigma_i$ so that
    \[ \sigma = \sigma_1, \ldots, \sigma_n = \hat{\sigma}, \] and for any
    $i$ the group $G_i = \langle \sigma_i, \sigma_{i+1} \rangle$ is a
    finite group so that $\Sigma_g / G_i$ is a torus with four
    marked points. In that case we also say that $\sigma, \hat{\sigma}$ are a \emph{good pair}.
  \end{enumerate}
\end{prop}
In the proof we need the notion of \emph{Humphries generators} for the
mapping class group.  We refer the reader to the textbook
\cite[Chapter~4]{Primer} for a detailed discussion, and only recall
the definition for convenience. Namely, a Humphries generating set for
the mapping class group of a genus $g$ surface consists of Dehn twists
about curves\footnote{In the terminology of \cite[Theorem~4.14]{Primer}, also referring to \cite[Figure~4.5]{Primer}, the curves $\alpha_1, \ldots, \alpha_{2g}$ are the curves $m_1, a_1, c_1, a_2, c_2, \ldots, c_{g-1}, a_g$, and the curve $\alpha_{2g+1}$ is the curve $m_2$.} $\alpha_i, i=1,\ldots, 2g+1$ so that
\begin{itemize}
\item $\alpha_1, \ldots, \alpha_{2g}$ form a chain,
i.e. $\alpha_i, \alpha_j$ intersect in one point if $|i-j|=1$, and are
disjoint otherwise.
\item $\alpha_{2g+1}$ is disjoint from all $\alpha_i$ except
  $\alpha_4$, which it intersects in a single point.
\end{itemize}
The crucial result \cite[Theorem~4.14]{Primer} is that Dehn twists about any such set of curves
generate the mapping class group.
\begin{proof}[{Proof of Proposition~\ref{prop:ls-islands}}]
  When $g$ is even this is \cite[Theorem 5.3]{LS-connectivity}. The
  case of odd genus is a fairly straightforward modification which is
  below.

  The strategy is as follows.  We show that for $f_1,\ldots,f_n$ a
  suitably chosen generating set for $\Mcg(\Sigma_g)$ and $\sigma$ a
  suitably chosen involution we have that $\sigma, f_i\sigma f_i^{-1}$
  are a good pair. Since whenever $\sigma, \sigma'$ are good pair,
  $g\sigma g^{-1}$ and $g\sigma' g^{-1}$ are as well, we have that by
  induction of the word length in $f_1,\ldots,f_n$, $\sigma$ can be joined to
  $f\sigma f^{-1}$ for any mapping class $f$.

  \smallskip To construct $\sigma$ and $\sigma'$, we use
  the following setup (compare Figure~\ref{fig:ring}).
  \begin{figure}[h!]
    \centering
    \includegraphics[width=0.6\textwidth]{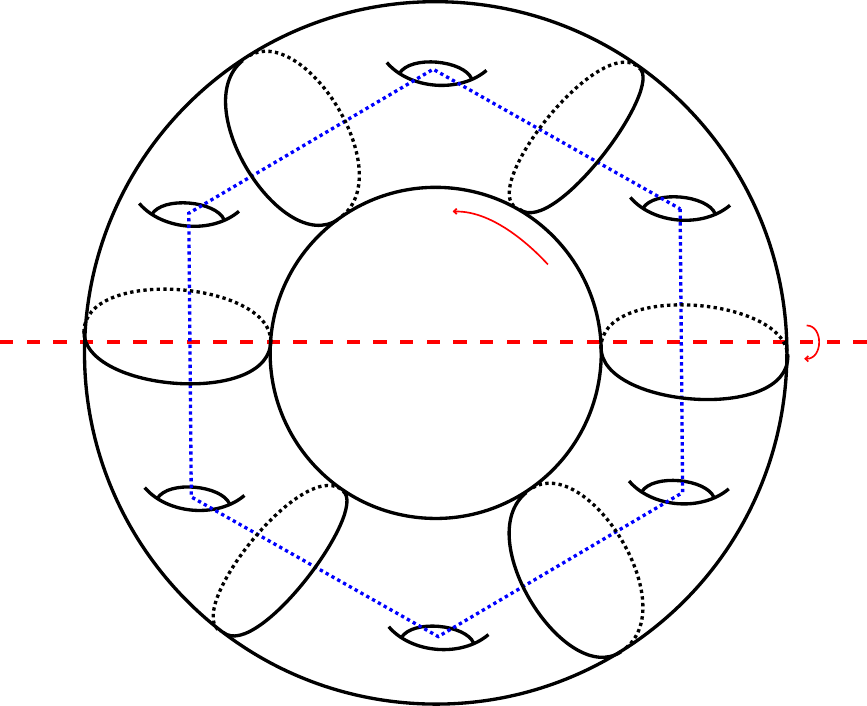}
    \caption{The setup for Proposition~\ref{prop:ls-islands}: realising the dihedral group action}
    \label{fig:ring}
  \end{figure}
  We realise the surface $S$ of genus $2k+1$ as a union
  \[ S = \bigcup_{i=0}^{2k-1} H_i \] where each $H_i$ is a torus with
  two boundary components, and the two boundaries of $H_i$ are glued
  to $H_{i+1}, H_{i-1}$ in a ring (compare
  Figure~\ref{fig:ring}). Denote by $\delta_0, \ldots, \delta_{2k-1}$
  the boundary curves of the $H_i$, so that $\partial H_i = \delta_i
  \cup \delta_{i+1}$. The dihedral group of order $4k$ embeds into the
  mapping class group of $S$, generated by an order $2k$ element $r$
  and an order $2$ element $\sigma$. We have that $r(H_i) =
  H_{i+1}, r(\delta_i) = \delta_{i+1}$ (where indices are taken mod
  $2k$), and $\sigma$ can be described in the following way: 
  the curves $\delta_0, \delta_k$
  cut $S$ into two subsurfaces $S_+, S_-$, each of which has genus
  $(g-1)/2$ and has two boundary components. The involution $\sigma$
  will exchange $S^+$ and $S^-$ and fix both boundary components of
  $S^+$ setwise.

  Intuitively, we imagine $S$ as a symmetric, thickened $2k$-gon in
  three-space, with a torus in each corner. The element $r$ then rotates
  the $2k$-gon by $\pi/k$ around its center, while $\sigma$ rotates by $\pi$
  about an axis through $\delta_0, \delta_k$ (compare Figure~\ref{fig:ring}).

  \smallskip We then define $\sigma' = r\sigma r^{-1}$. 
  We claim that $\Sigma_g/\langle \sigma, \sigma' \rangle$ is
  a torus with four marked points. Indeed, $\langle \sigma, \sigma' \rangle$
  contains $r^2$ (recall that $\sigma, r$ generate a dihedral group),
  and thus
  \[ H_0\cup H_1 \to \Sigma_g/\langle \sigma, \sigma' \rangle \]
  is already surjective. Since $\sigma'$ exchanges $H_0$ and $H_1$,
  even 
  \[ H_0 \to \Sigma_g/\langle \sigma, \sigma' \rangle \] is already
  surjective. In fact, $\Sigma_g/\langle \sigma, \sigma' \rangle$ is
  obtained from $H_0$ by identifying two halves of $\delta_0$ with
  each other (via the action of $\sigma$) and identifying two halves
  of $\delta_1$ with each other (via the action of $\sigma'$).  This
  shows that $\Sigma_g/\langle \sigma, \sigma' \rangle$ is indeed a
  torus with four marked points (coming from the fixed points of
  $\sigma, \sigma'$ in $H_0$).

  \begin{figure}[h!]
    \centering
    \includegraphics[width=0.6\textwidth]{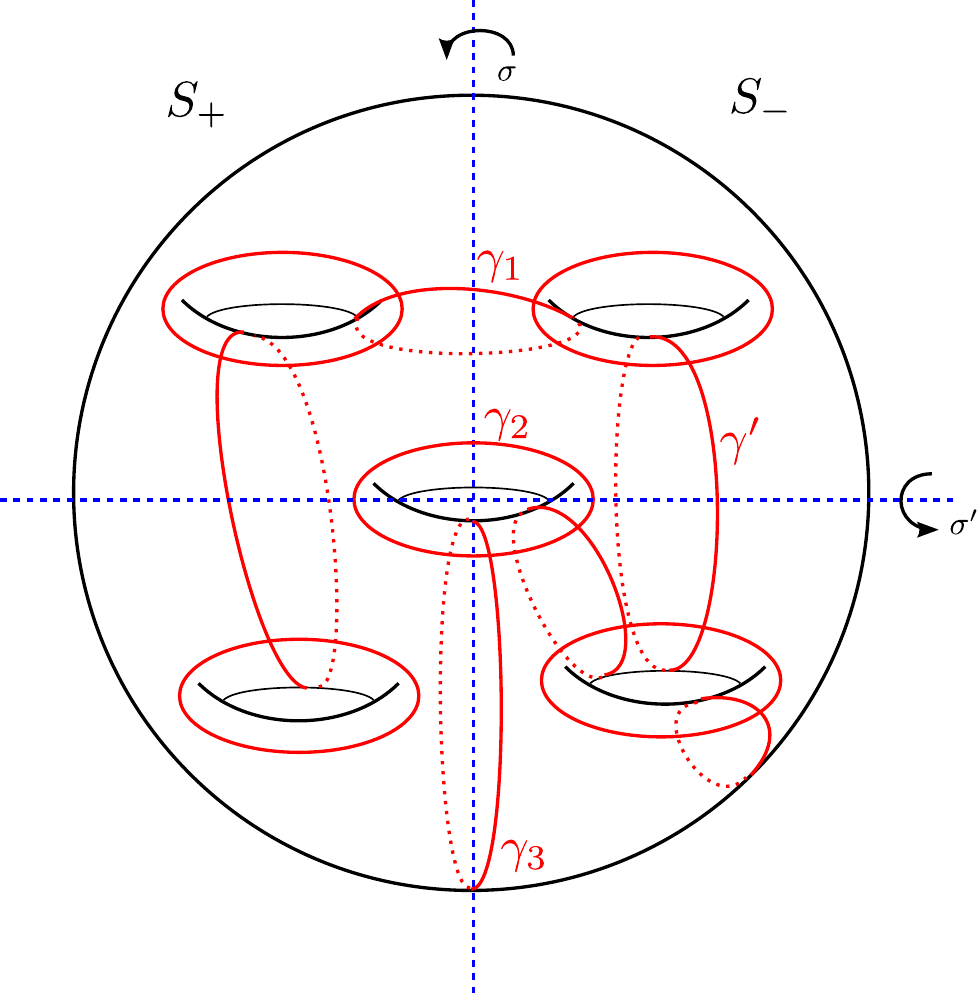}
    \caption{The setup for Proposition~\ref{prop:ls-islands}: $\gamma_1, \gamma_2, \gamma_3$ are the curves which are not in $S^\pm$ and they are invariant under $\sigma$. The curve $\gamma'$ is invariant under $\sigma'$.}
    \label{fig:involutions}
  \end{figure}

  Next, we claim that there are simple closed curves $\alpha_i$ with
  the following properties:
  \begin{enumerate}[a)]
  \item Dehn twists about the $\alpha_i$ form a (Humphries) generating
    set for the mapping class group of $\Sigma_g$.
  \item Each $\alpha_i$ is either contained in one of the $S^\pm$, or
    is invariant under $\sigma$.
  \item If $\alpha_i \subset S^\pm$, then it is nonseparating in that
    subsurface.
  \item There is one $\alpha_{j_0}$ which is contained in $S^-$ and which
    is invariant under $\sigma'$.
  \end{enumerate}
  That such a set of curves exists is an exercise using Figure~\ref{fig:involutions}.

  Now, from property c) we get the following:
  \begin{equation}
    \label{eq:commute-map}
    \forall \alpha_i\mbox{ not invariant under }\sigma \quad\exists\phi_i \in \mathrm{Mcg}(\Sigma_g): [\phi_i, \sigma]=1, \phi_i(\alpha_{j_0})=\alpha_i.
  \end{equation}
  Namely, suppose first that $\alpha_i \subset S^-$. Then, since both $\alpha_i, \alpha_{j_0}$ are nonseparating in $S^-$, there is a mapping class $f$ of $S^-$
  fixing $\partial S^-$ which sends $\alpha_{j_0}$ to $\alpha_i$. Extend $f$ to
  a mapping class $\phi_i$ of $S$ by setting it to be $\sigma f \sigma$
  on $S^+$. This has the desired property. In the case where $\alpha_i \subset S^+$, we start with $f$ which sends $\sigma\alpha_i$ to $\alpha_{j_0}$ as above,
  and let $\phi_i$ be $\sigma f$ on $S^-$ and $f \sigma$ on $S^+$.

  \smallskip
  We claim that for any of the Humphries generators $T=T_{\alpha_i}$
  we can connect $\sigma$ to $T\sigma T^{-1}$ with a path as in ii) of the statement of the Proposition. 

  For twists about curves
  $\alpha_i$ which are invariant under $\sigma$ there is nothing to
  show, as such twists commute with $\sigma$, and therefore the trivial
  path connects $\sigma$ and $T_{\alpha_i}\sigma T_{\alpha_i}^{-1} = \sigma$.
  If $\alpha_i$ is not
  invariant, let $\phi_i$ be the mapping class guaranteed by~(\ref{eq:commute-map}). We claim that
  \[ \sigma_1 = \sigma, \]
  \[ \sigma_2 = \phi_i\sigma'\phi_i^{-1}, \]
  \[ \sigma_3 = T_{\alpha_i}\sigma T_{\alpha_i}^{-1} \] is a path as desired. To begin with, note
  that \[ G_1 = \langle \sigma, \phi_i\sigma'\phi_i^{-1} \rangle =
  \langle \phi_i \sigma \phi_i^{-1}, \phi_i\sigma'\phi_i^{-1} \rangle
  = \phi_i \langle \sigma, \sigma'\rangle \phi_i^{-1},\] since $\phi_i$ commutes with $\sigma$. As
  by assumption $\sigma, \sigma'$ is a good pair, $G_1$ is a group as desired.

  Next, observe that
  \[ \phi_i T_{\alpha_{j_0}} \phi_i^{-1} = T_{\phi_i\alpha_{j_0}} = T_{\alpha_i} \] and therefore
  \[ [T_{\alpha_i}, \phi_i\sigma'\phi_i^{-1}] =
  [\phi_iT_{\alpha_{j_0}}\phi_i^{-1}, \phi_i\sigma'\phi_i^{-1}] =
  \phi_i[T_{\alpha_{j_0}}, \sigma']\phi_i^{-1} = 1, \] 
  since $\sigma'$ preserves $\alpha_{j_0}$ and therefore commutes with the Dehn twist about $\alpha_{j_0}$. 
  As $G_1$ is generated by a good pair, so is 
  \[ T_{\alpha_{j_0}}G_1T_{\alpha_{j_0}}^{-1} = \langle T_{\alpha_{j_0}}\sigma T_{\alpha_{j_0}}^{-1}, T_{\alpha_{j_0}}  \phi_i\sigma'\phi_i^{-1} T_{\alpha_{j_0}}^{-1}\rangle =  \langle T_{\alpha_{j_0}}\sigma T_{\alpha_{j_0}}^{-1}, \phi_i\sigma'\phi_i^{-1} \rangle = \langle \sigma_3, \sigma_2 \rangle. \]
  Hence, $\sigma_1, \sigma_2, \sigma_3$ is indeed a path as desired.
\end{proof}

For the remainder of this section, we fix $\sigma$ to be as in the conclusion of Proposition~\ref{prop:ls-islands}. Say that a foliation $F$ is \emph{lifted torus good}, if $F$ is the lift of a torus good foliation on $\Sigma_g/\hat{\sigma}$ for $\hat{\sigma}$ a conjugate of $\sigma$
in $\Mcg(S)$ (possibly by the identity).

The following will replace Lemma~\ref{lem:island-1}.
\begin{lem}\label{lem:island-2}
  Suppose that $F, F'$ are lifted torus good. Then there are
  \begin{enumerate}
  \item Involutions $\sigma_1,\ldots, \sigma_N$, which are conjugate to $\sigma$,
  \item Abelian differentials $\omega_i, i=1,\ldots, N$ on $\Sigma_g$,
  \end{enumerate}
  so that the following hold:
  \begin{enumerate}[i)]
  \item For any $i$, the group
    $\langle \sigma_i, \sigma_{i+1} \rangle$ is finite and
    $T_i = \Sigma_g / \langle \sigma_i, \sigma_{i+1} \rangle$ is a torus
    with four marked points.
  \item The differential $\omega_i$ is a lift of a torus good differential on the
    torus $T_i$ (with marked points).
  \end{enumerate}
\end{lem}
\begin{proof}
  Suppose that $F$ is a lift of a foliation on $\Sigma_g/\sigma$ and
  $F'$ is a lift of a foliation on $\Sigma_g/\sigma'$. Apply
  Proposition~\ref{prop:ls-islands} to $\sigma, \sigma'$ to find the
  involutions $\sigma_i$ with property i). The differentials
  $\omega_1, \omega_N$ are chosen to be the ones defining $F, F'$; the other
  $\omega_i$ can be chosen as arbitrary lifts of torus good differentials on $T_i$.
\end{proof}

Finally, the following will replace Theorem~\ref{thm:contracted-paths}. 
\begin{thm}\label{thm:contracted-paths-2}
  Suppose that $(\tau_n)$ is
  a full splitting sequence in the direction of a uniquely ergodic
  foliation $E$, and let $f_m$ be an associated $\Mcg$-sequence.

  Fix an essential type $k$ and let $F, F' \in \mathcal{U}^{(k)}$ be two lifted torus good
  foliations lifted from covers $\Sigma_g/\sigma, \Sigma_g/\sigma'$.
  Assume that
  \begin{description}
  \item[$(\ast)$] $\mathcal{U}^{(k)}$ contains every foliation which is a lift of the cover defined by $\Sigma_g/\sigma, \Sigma_g/\sigma'$.
  \end{description}
  Then for any $n$ there is a number $m_0$ with the following
  property.  Suppose that $m>m_0$ and that $k_m=k$.  Then there is an
  path $\gamma$ connecting $F$ to $F'$, so that $f_m\gamma$ is
  completely contained in $U_n(\tau, E)$, and consists only of
  cobounded foliations. Moreover, given any finite set $S$ we may assume that $f_m\gamma$ does not intersect $S\setminus \{f_mF,f_mF'\}$. 
\end{thm}

\begin{proof} 
  Suppose that $F, F'$ are given as in the theorem. First, apply
  Lemma \ref{lem:island-2} to obtain a sequence of involutions
  $\sigma_1,...,\sigma_N$. We now have two sequences of covers
  \[ p_i: \Sigma_g \to \Sigma_g/\sigma_i \]
  and
  \[ t_i: \Sigma_g \to \Sigma_g/\langle\sigma_i,\sigma_{i+1}\rangle \]
  which are compatible in the sense that $t_i$ factors through both
  $p_i$ and $p_{i+1}$:
  \begin{center}
  	\begin{tikzcd}
  	& \Sigma_g\arrow[dl, "p_i"']\arrow[dr, "p_{i+1}"]\arrow[dd, "t_i"] & \\
  	\Sigma_g/\sigma_i\arrow[dr, dashed]  & & \Sigma_g/\sigma_{i+1}\arrow[dl, dashed] \\
  	&  \Sigma_g/\langle\sigma_i,\sigma_{i+1}\rangle &
  	\end{tikzcd}
  \end{center}
	Now for each $i$, let $\delta_i$ be a lift of
  a simple closed curve on the four times punctured torus
  $\Sigma_g/\langle\sigma_i,\sigma_{i+1}\rangle$ by the map $t_i$, and
  let $\mu_i$ be a lift of a simple closed curve from $\Sigma_g/\sigma_i$ by
  the map $p_i$. We will next construct lifted torus good foliations $B_i, I_j^+, I_j^-$,
  and the desired path as a concatenation
  \[ \gamma = \gamma_1^+ \ast \overline{\gamma_2^-} \ast \gamma_2^0 \ast \gamma_2^+ \ast \dots
  \ast \overline{\gamma_{N-1}^-} \ast \gamma_{N-1}^0 \ast \gamma_{N-1}^+ \ast \overline{\gamma_N^-} \]
  where $\overline{\cdot}$ denotes the path with opposite orientation, and
  \begin{enumerate}
  \item $\gamma^+_j$ is a path starting in $I_j^+$, and ending in $B_{j+1}$, 
  \item $\gamma^0_j$ is a path starting in $I_j^-$, and ending in $I_j^+$,
  \item $\gamma^-_j$ is a path starting in $I_j^-$, and ending in $B_{j-1}$, 
  \end{enumerate}
  All $\gamma^\ast_j$ will be produced by using
  Corollary~\ref{cor:passing-to-branched-covers}.
  
  Namely, put $I_0^+ = F, I_N^- = F'$. For the remaining 
  $B_i, I_j^+, I_j^-$ choose lifted torus good foliations (for the covers $p_i, t_j, t_j$
  respectively)
  close enough to $\delta_i,
  \mu_j$ so that we can apply Corollary~\ref{cor:passing-to-branched-covers}. 
  (to
  Theorem~\ref{thm:contracted-paths}), in the version without $(\ast)$. Note that this closeness depends on $m\geq m_0$ (and not just $m_0$).
  This can e.g. be achieved by starting with any lifted torus good foliations, and
  Dehn twisting them about the $\delta_i, \mu_j$. Note that we can produce disjoint $B_i,I_j^+,I_j^-$ by twisting different numbers of times, and produce distinct paths as in Corollary~\ref{cor:tg-connected}.
  
  Now, for $\gamma_1^+$ and $\overline{\gamma_N^-}$, we will apply Corollary~\ref{cor:passing-to-branched-covers} with $(\ast)$. Indeed, as both $I_0^+$ and $B_1$ are both given by the same lifting maps $(\ast)$ in the statement of this theorem gives the assumption of $(\ast)$ in Corollary~\ref{cor:passing-to-branched-covers}. (Similarly fo $I_n^-$ and $B_{N-1}$.) Note that by Corollary~\ref{cor:passing-to-branched-covers}, these paths can be chosen to overlap only at $F,F'$.
\end{proof}

With this in place, we can finish the proof of
Theorem~\ref{thm:main-closed} exactly 
as in the case of
Theorem~\ref{thm:main-punctured}. 

In fact, the proof shows something a little bit stronger, which will be
useful to show local path-connectivity.
\begin{cor}\label{cor:good-path-nesting}
  Suppose $\tau$ is a train track carrying a uniquely ergodic
  foliation $E$, and suppose that $\tau_n$ is a splitting
  sequence in the direction of $E$. Then for any $n$ there is a
  $m=m(\tau,n, E)$ with the following property. If $E'$ is any
  uniquely ergodic foliation contained in $U_m(\tau, E)$, then there is
  a path of uniquely ergodic laminations connecting $E'$ to $E$
  completely contained in $U_n(\tau, E)$.
\end{cor}
\begin{proof}
  In the case of a punctured surface, i.e. Theorem~\ref{thm:main-punctured}, all bounds on
  $m$ come from applying Proposition~\ref{prop:curves-get-contracted} or~\ref{prop:apply-pseudo}
  within the proof of Theorem~\ref{thm:contracted-paths}.
  By Lemma~\ref{lem:following-sequences} we can choose these bounds to be independent of 
  the actual foliation guiding the splitting sequence, as long as the foliation is contained in
  $U_k(\tau, E)$ for $k$ large enough. 
  The bounds in Theorem~\ref{thm:main-closed} come from applying Theorem~\ref{thm:contracted-paths} and
  its Corollary~\ref{cor:passing-to-branched-covers}, and so the same is true there.
\end{proof} 

\section{Local Path Connectivity}
In this section, we improve the Theorem from the last section to the
following.

\label{sec:localpath}
\begin{thm}\label{thm:local-path-conn}
  If $g\geq 5$ or $g\geq 2, p\geq 1$, the set of uniquely ergodic foliations on $S_{g,p}$ is locally
  path-connected. 
\end{thm}

Given a uniquely ergodic foliation $F$ and a full 
splitting sequence $(\tau_s)_s$ towards $F$. For any $n$, we let $m(\tau, n,F)$
the number guaranteed by Corollary~\ref{cor:good-path-nesting}.
Define
\[ \mathcal{G}_n(\tau,
F)=U_{m(\tau,n,F)}(\tau, F). \]
Corollary~\ref{cor:good-path-nesting} guarantees that for any $F' \in
\mathcal{G}_n(\tau,F)$ there exists a path $P_{F,F'}$ of cobounded
foliations joining $F$ to $F'$, which is contained in $U_n(\tau,F)$.

Let $\hat{\mathcal{G}}_n(\tau, F)$ be the intersection of
$\mathcal{G}_n(\tau, F)$ with the set of uniquely ergodic
foliations.  For any point $p \in P_{F,F'}$, we can define a neighbourhood
\[\mathcal{G}_n(p,\tau)\]
as above, i.e. with the property that $p$ can be joined to any $p' \in \mathcal{G}_n(p,\tau)$ by a
path of cobounded foliations which is contained in $U_n(\tau, F)$.

Also observe that
\begin{equation}
  \label{eq:compatible-u-nbhds}
  U_i(\tau, p) \subset U_i(\tau, F) 
\end{equation}
for all $i \leq n$.

Define
\[ N^{(1)}(F,n):=\bigcup_{F' \in \hat{\mathcal{G}}_n(\tau, F)}\bigcup_{p \in P_{F,F'}}\mathcal{G}_n(p,\tau).\]
Inductively, put
\[N^{(r+1)}(F,n)=\bigcup_{p \in N^{(r)}(F,n)} N^{(1)}(p,n).\] 

Also observe that we have $N^{(r)}(F,n) \subset U_n(\tau,F)$
by Equation~(\ref{eq:compatible-u-nbhds}), whenever 
$F' \in \mathcal{G}_n(\tau, F)$.

\begin{prop}\label{prop:inductive-path-construction}
  Any point in $N^{(r)}(F,n)$ is connected to $F$ by a
  path of uniquely ergodic foliations, which is contained in in
  $N^{(r+1)}(F,n)$.
\end{prop}
\begin{proof}
We prove this by induction.

\textit{Base case}:  If $p \in N^{(1)}(F,n)$ then we can connect it to $F$ by a path in $N^{(2)}(F,n)$. 

\begin{proof} If $p\in P_{F,F'}$ this is
  obvious. Otherwise $p\in \mathcal{G}_n(\hat{p},\tau)$ for some
  $\hat{p} \in P_{F,F'}$ where
  $F' \in \mathcal{G}_n(\tau, F)$. By definition we have
  that there exists a path of cobounded foliations contained in
  $\mathcal{G}_n(\hat{p},\tau)$ connecting $p$ to
  $\hat{p}$. Concatenating this with the segment of
  $P_{F,F'}$ connecting $\hat{p}$ to $F$ connects
  $p$ to $F$.  The first segment of the path is in
  $N^{(1)}({\hat{p}},n)$ and so the whole thing is in
  $N^{(2)}(F,n)$.
\end{proof}

\textit{Inductive step}: Assume $p \in N^{(r)}(F,n)$ and that any point in $N^{(r-1)}(F,n)$ is  connected to $F$ by a path of cobounded foliations in $N^{(r)}(F,n)$. We will now show that $p$ is path connected by cobounded foliations in $N^{(r+1)}(F,n)$ to $F$.

\begin{proof} Because $p \in N^{(r)}(F,n)$ we know (by
  definition of $N^{(r+1)}$) $p \in N^{(1)}({\hat{p}},n)$ for some
  $\hat{p} \in N^{(r-1)}(F,n)$. By the base case of induction applied to $\hat{p}$ it is
  connected to $\hat{p}$ by a path in
  $N^{(2)}({\hat{p}},n)=\cup_{p'\in
    N^{(1)}(\hat{p},n)}N^{(1)}({p'},n).$ This is contained in
  $\cup_{p'\in N^{(r)}(F,n)}N^{(1)}({p'},n)
  =N^{(r+1)}(F,n)$.  To finish linking $p$ to $F$ we use
  our inductive assumption to link $\hat{p}$ to $F$ by a path in
  $N^{(r-1+1)}(F,n)$.
\end{proof}
\end{proof}
\begin{cor}
  For any uniquely ergodic foliation $F$, and any $n$, the set
  \[ \left( \bigcup_{r \geq 1}N^{(r)}(F, n) \right) \cap \UE \]
  is an open neighbourhood of $F$ in $\UE$, which is path-connected and contained in $U_n(\tau, F)$.
\end{cor}
\begin{proof}
  The set is open as a union of open subsets. It is contained in
  $U_n(\tau, F)$, since all $N^{(r)}(F, n)$ have this property. It is
  path-connected by
  Proposition~\ref{prop:inductive-path-construction}.
\end{proof}
By Corollary~\ref{cor:u-nest}, the $U_n(\tau, F)$ are a basis for neighbourhoods of $F$ in $\UE$, and
thus this finishes the proof of Theorem~\ref{thm:local-path-conn}.

\bibliographystyle{math}
\bibliography{lampaths}

\end{document}